\documentclass[12pt,reqno,a4paper]{amsart}
\usepackage[top=3.5cm,bottom=3.5cm,left=2.5cm,right=2.5cm,heightrounded,bindingoffset=0mm]{geometry}
\usepackage[utf8]{inputenc}
\usepackage[T1]{fontenc}
\usepackage[english]{babel}
\usepackage{microtype}
\usepackage{lmodern}
\usepackage{amsmath,amsthm,amssymb,mathrsfs,thmtools}
\usepackage{enumerate} 
\usepackage{graphicx}
\usepackage{svg}
\usepackage{color}
\usepackage{tikz}

\newcommand{\bq}{\begin{equation}}
\newcommand{\eq}{\end{equation}}
\newcommand{\bqa}{\begin{eqnarray*}}
\newcommand{\eqa}{\end{eqnarray*}}

\DeclareMathOperator{\supp}{supp}

\DeclareSymbolFont{pletters}{OT1}{cmr}{m}{sl}
\DeclareMathSymbol{s}{\mathalpha}{pletters}{`s}

\def\Jac{{\rm D}}

\def\mez{\frac{1}{2}}

\def\Rr{\mathbb{R}}
\def\Nn{\mathbb{N}}
\def\Zz{\mathbb{Z}}

\def\E{\mathbb E}

\def\L1{\mathcal{L}^{(1)}}
\def\L2{\mathcal{L}^{(2)}}
\def\L3{\mathcal{L}^{(3)}}

\def\T{\mathbb{T}}

\numberwithin{equation}{section}

\def\Op{\operatorname{Op}}

\newcommand{\an}[1]{\langle #1 \rangle}
\newcommand{\ring}[1]{\mathring{#1}}
\def\dVolg{\mathrm{d}\!\operatorname{Vol}_g}

\title[No quasi-invariance under transport flows]{Obstruction to quasi-invariance of Gaussian measures under transport flows on Riemannian Manifolds}
\date{\today}

\author[Anne-Sophie de Suzzoni]{Anne-Sophie de~Suzzoni$^{1,2}$}
\address{
$^1$Laboratoire de Mathématiques et de Modélisation d'\'Evry (LaMME)\\
Université d'\'Evry\\
23 Bd François Mitterrand, 91000 Évry-Courcouronnes, France
}
\address{$^2$Institut universitaire de France (IUF)}
\email[A.-S. de Suzzoni]{anne-sophie.desuzzoni@univ-evry.fr}

\author[Mickaël Latocca]{Micka\"el Latocca$^1$}
\email[M. Latocca]{mickael.latocca@univ-evry.fr}

\makeatletter
\def\l@subsection{\@tocline{2}{0pt}{2.5pc}{5pc}{}}
\makeatother
\makeatletter
\def\@tocline#1#2#3#4#5#6#7{\relax
  \ifnum #1>\c@tocdepth 
  \else
    \par \addpenalty\@secpenalty\addvspace{#2}%
    \begingroup \hyphenpenalty\@M
    \@ifempty{#4}{%
      \@tempdima\csname r@tocindent\number#1\endcsname\relax
    }{%
      \@tempdima#4\relax
    }%
    \parindent\z@ \leftskip#3\relax \advance\leftskip\@tempdima\relax
    \rightskip\@pnumwidth plus4em \parfillskip-\@pnumwidth
    #5\leavevmode\hskip-\@tempdima
      \ifcase #1
       \or\or \hskip 1em \or \hskip 2em \else \hskip 3em \fi%
      #6\nobreak\relax
    \dotfill\hbox to\@pnumwidth{\@tocpagenum{#7}}\par
    \nobreak
    \endgroup
  \fi}
\makeatother

\makeatletter
\renewcommand{\@@and}{\&}
\makeatother

\usepackage{hyperref}
\hypersetup{colorlinks={true},linkcolor={blue},citecolor=blue}
\makeatletter
\AddToHook{cmd/appendix/before}{\def\cref@section@alias{appendix}}
\makeatother
\usepackage[nameinlink,noabbrev,capitalise]{cleveref}

\theoremstyle{plain}
\newtheorem{theo}{Theorem}[section]
\newtheorem{prop}[theo]{Proposition}
\crefname{prop}{Proposition}{Propositions}
\newtheorem{lemm}[theo]{Lemma}
\crefname{lemm}{Lemma}{Lemmas}
\newtheorem{coro}[theo]{Corollary}

\newtheorem{assumption}[theo]{Assumption}
\newtheorem{defi}[theo]{Definition}
\theoremstyle{definition}
\newtheorem{rema}[theo]{Remark}

\thanks{\textit{Acknowledgements.} We are grateful to Nikolay Tzvetkov for his insightful suggestions at an early stage of this project. \\
AS de Suzzoni is supported by Soutien recherche et attractivité – PR et DR" 2026 from Université Paris-Saclay and by ANR No-Limit ANR-25-CE40-1380}

\makeatletter
\renewcommand{\@settitle}{%
  \begin{center}
    \baselineskip16\p@\relax
    \LARGE\bfseries
    \@title
    \par
  \end{center}%
}
\makeatother

\makeatletter
\patchcmd{\@setauthors}
  {\MakeUppercase{\authors}}
  {\large\authors}
  {}
  {\PackageWarning{mydocument}{Failed to patch \string\@setauthors}}
\makeatother

\begin{document}
\newcommand{\annso}[1]{{\color{orange} \textbf{AS:} #1}}
\newcommand{\mickael}[1]{{\color{teal} \textbf{M:} #1}}

\begin{abstract}
We prove an obstruction to quasi-invariance of Gaussian fields under divergence free transport flows on Riemannian manifolds.  For the centered Gaussian field with covariance $(1-\Delta_g)^{-\alpha}$, $\alpha>\frac{d}{2}$, quasi-invariance under the transport flow is equivalent to invariance. This happens precisely when the flow acts by isometries, or equivalently when the underlying vector field is Killing. The proof leverages the Feldman--H\'ajek theorem and utilizes localized pseudo-differential computations. On $\mathbb R^d$ this leaves only rigid motions, while on flat tori this leaves only translations.
\end{abstract}

\hrule
\vspace{0.5cm}
\maketitle
\hrule 
\vspace{0.4cm}

\section{Introduction}

Consider a smooth, connected, boundaryless and complete Riemannian manifold $(M,g)$ of dimension~\(d\geq 2\). We denote its tangent bundle~\(TM\).

Let~\(u : \Rr\times M \rightarrow TM\) be a~\(L^1_{\rm loc}(\Rr_t,\mathcal C_x^{1,\varepsilon}(M))\) divergence-free vector field.   
It has two variables: time~\(t\in\Rr\) and space~\(x\in M\). We consider the transport equation
\begin{equation}\label{eq:transport}
    \partial_t \Omega + u \cdot \nabla \Omega = 0
\end{equation}
where the unknown~\(\Omega : \Rr \times M \rightarrow \Rr\) is a scalar field; and $u\cdot \nabla = u^j\partial_j$ in coordinates.

Let $\mu_{\alpha}$ be a centered Gaussian measure on~\(L^2_{\mathrm{loc}}(M)\) with covariance operator 
$(\operatorname{id}-\Delta_g)^{-\alpha}$, where $-\Delta_g$ is the Laplace--Beltrami operator and~\(\alpha > \frac{d}{2}\). We recall that this means that for all test functions~\(f,g\), we have
\[
\int \an{f,\phi}_{L^2(M)} \an{\phi,g}_{L^2(M)} \mathrm{d}\mu_\alpha (\phi) = \an{f,(1-\Delta_g)^{-\alpha} g}_{L^2(M)}.
\]
We give an explicit realization of such measures in~\cref{subsec:Result} for compact manifolds and the Euclidean space; and refer to~\cref{As.FH} for generic complete Riemannian manifolds.

We are concerned with solving \eqref{eq:transport} for data taken~\(\mu_\alpha\) almost surely.

Equation~\eqref{eq:transport} may be solved by the method of characteristics. For all~\(x\in M\) and~$t_0\in \Rr$, let~\(t\mapsto X_{t,t_0}(x)\) be the solution to the ODE
\begin{equation}\label{eq:defX}
\left \lbrace{\begin{array}{rcl}
    \partial_t X_{t,t_0}(x) &=& u(t,X_{t,t_0}(x)), \\
    X_{t_0,t_0}(x) &=& x.
\end{array}} \right.
\end{equation}
We write $X_t=X_{t,0}$. Since $u \in L^1_{\rm loc}(\mathbb{R},\mathcal{C}^{1,\varepsilon})$ in particular for all $T>0$ we have 
\[
    \int_0^T \left(\|u(\tau)\|_{L^{\infty}(M)} + \|\nabla u(\tau)\|_{L^{\infty}(M)}\right)\mathrm{d}\tau < \infty.
\]
It follows from Picard--Lindelöf's theorem and the completeness of $(M,g)$ that the flow $X_{t,t_0}$ (and thus~\(X_t\)) is defined globally in time and space. Moreover, if $t$ is sufficiently close to $t_0$, then the map~\(x\mapsto X_{t,t_0}(x)\) is a~\(\mathcal C^{1,\varepsilon}\) diffeomorphism. Using that for any~\(s\in (t_0,t)\) we have
\[
X_{t,t_0} = X_{t,s} \circ X_{s,t_0} , 
\]
we get that~\(X_{t,t_0}\) is a~\(\mathcal C^{1,\varepsilon}\) diffeomorphism at all times bigger than~\(t_0\), and with a similar reasoning, this applies to times smaller than~\(t_0\).

We recall that the Cauchy problem
\[
\left \lbrace{\begin{array}{c}
\partial_t \Omega + u\cdot \nabla \Omega = 0 \\
\Omega (t_0) = \Omega_0
\end{array}} \right.
\]
is solved by remarking that
\[
   \frac{\mathrm{d}}{\mathrm{d}t}\Omega(t,X_{t,t_0}(x)) = \partial_t \Omega (t,X_{t,t_0}(x)) + \partial_t X_{t,t_0}(x) \cdot \nabla \Omega(t,X_{t,t_0}(x)) = 0
\]
and therefore
\begin{equation}\label{eq:defsemigroup}
    \Omega(t,x) = \Omega_0(X_{t,t_0}^{-1}(x)).
\end{equation}
We denote~\(S(t,t_0)\) the composition with~\(X_{t,t_0}^{-1}\) such that~\(\Omega(t) = S(t,t_0) \Omega_0\). Note that the~\(L^2(M)\) norm is \emph{a priori} preserved by~\eqref{eq:transport}. 

We denote $S(t)$ the composition with $X_{t}^{-1}$. Note that~\(S(t) = S(t,t_0) \circ S(t_0)\).

We are interested in the quasi-invariance of~\(\mu_\alpha\) under the action of the flow~\(S(t)\). In other words, we wish to characterize the absolute continuity of~\(\mu_\alpha^t : = S(t)_* \mu_\alpha\) with respect to~\(\mu_\alpha\).

\subsection{Quasi-invariance of Gaussian measures in PDE}

The question of quasi-invariance of Gaussian measures under the flow of nonlinear PDEs has received considerable attention in recent years, particularly in the dispersive Hamiltonian setting. We mention here only a selection of representative results. The systematic study of this problem was initiated by Tzvetkov in \cite{tzvetkov2015}. Subsequent developments have relied on several related mechanisms, including modified energies, normal-form reductions, nonlinear smoothing estimates, and quantitative control of Radon--Nikodym densities; see, among many others,
\cite{ohTzvetkov2017bis,ohSosoeTzvetkov2018,ohTzvetkov2020,forlanoTrenberth2019,planchonTzvetkovVisciglia2020,ohTsutsumiTzvetkov2019,debusscheTsutsumi2021,ohSeong2021,forlanoSeong2022,planchonTzvetkovVisciglia2023}.
In several more recent low-regularity problems, these techniques have been complemented by probabilistic and variational arguments, including applications of the Boué--Dupuis formula
\cite{gunaratnamOhTzvetkovWeber2022,forlanoTolomeo2025a,knezevitch2025a,knezevitch2025b,knezevitch2025c,sunTzvetkov2025,tolomeoVisciglia2025}.
Related ideas have also entered the analysis of geometric equations, notably the binormal flow through the Hasimoto transform and the associated cubic nonlinear Schrödinger equation \cite{banicaLucaTzvetkovVega2025}.

Much less is known about mechanisms leading to failure of quasi-invariance. In \cite[Theorem~1.6]{ohSosoeTzvetkov2018}, it is shown that, for $s>\frac{1}{2}$ and every nonzero time, the Gaussian measure under consideration is not quasi-invariant under the dispersionless pointwise flow
$i\partial_tu=|u|^2u$. For the cubic Szeg\H{o} equation, \cite{coeTolomeo2025} establishes a transition between quasi-invariance and singularity: the relevant Gaussian measure is quasi-invariant for $s>1$, whereas, for $\mez<s<1$ with $s\neq\frac{3}{4}$, the initial and transported measures are mutually singular for almost every time.

Comparatively, few results study the preservation of Gaussian structure by fluid equations. In vorticity formulation, spatial white noise on $\mathbb T^2$ is known to be stationary for suitable weak solutions of the two-dimensional Euler equation \cite{albeverioCruzeiro1990,flandoli2018}. The work \cite{deSuzzoni2022} establishes an asymptotic propagation of independence and Gaussian moments of Fourier modes in a large-torus asymptotic. On the other hand, \cite{bedrossianLatocca2026} proved non-invariance for a broad class of sufficiently regular Gaussian measures.

We emphasize the fact that most of these results are on compact settings. Part of the originality of the present work is that it explains how to derive global information from local information by localization of the Feldman–Hájek criterion and its compatibility with pseudo-differential calculus on a possibly non-compact manifold.

\subsection{Main result}\label{subsec:Result}

In the rest of this paper, we make the following assumption on~\(M\) and~\(\alpha\).

\begin{assumption}\label{as}
For all~\(\theta_1\) and~\(\theta_2\) in~\(\mathcal C_c^\infty (M)\) such that~\(\supp(\theta_1)\) is included in the interior of \( \theta_2^{-1}(\{1\})\), we have that operator
\[
    \theta_1 (1-\Delta_g)^{\frac{\alpha}{2}} (1-\theta_2) 
\]
is Hilbert--Schmidt on~\(L^2(M)\).
\end{assumption}

\begin{rema}
This assumption is satisfied in many cases. First, if~\(\alpha \in 2\Nn\); second, if~\(M\) is compact; third, if $M$ is the Euclidean space endowed with the Euclidean metric. We refer to \cref{lemma-pseudodiff-power}~(\textit{iv}) for the proofs.
\end{rema}

\begin{theo}[Generic Riemannian manifold case]\label{thm.riemannian}  Let $\alpha >\frac{d}{2}$  Assume that~\(u\) belongs to~\(L^1_{\mathrm{loc}}(\mathbb{R},\mathcal C^{\alpha}(M))\)\footnote{When $\alpha$ is not an integer, $\mathcal{C}^{\alpha}$ is understood as $\mathcal{C}^{\lfloor \alpha\rfloor, \alpha - \lfloor \alpha \rfloor}$}.
The following statements are equivalent: 
\begin{enumerate}[(i)]
    \item measure $\mu_{\alpha}$ is quasi-invariant under the flow of~\eqref{eq:transport};
    \item measure $\mu_{\alpha}$ is invariant under the flow of~\eqref{eq:transport};
    \item vector field~$u(t, \cdot)$ is a Killing vector field of $(M,g)$ for almost all $t$.
\end{enumerate}
\end{theo}

\begin{rema}
A Killing vector field is a vector field~\(v\) whose transport flow is a global isometry. It is characterised by the condition
\[
    \mathcal{L}_{v}g=0
\]
where~\(\mathcal L_{v}\) is the Lie derivative with respect to~\(v\). In coordinates, a Killing vector field is a vector field such that for all~\(\mu,\nu\), 
\[
    \nabla_{\mu} v_\nu + \nabla_\nu v_\mu = 0
\]
where we recall that the connection~\(\nabla\) writes in coordinates
\[
    \nabla_\mu v_\nu = \partial_\mu v_\nu  - \Gamma^\rho_{\;\; \mu \nu } v_\rho 
\]
with~\(\Gamma^\rho_{\;\; \mu \nu}\) the Christoffel symbols.

We refer to \cite[Proposition 8.1.1]{petersen}. Although this result is very standard, it is commonly written for a vector field that does not depend on time. We include a proof in coordinates in~\cref{app:Killing}.

Remark that Killing vector fields are divergence-free.
\end{rema}

\begin{rema} The Lie algebra of Killing vector fields of a smooth boundaryless connected Riemannian manifold of dimension $d$ has dimension at most $\frac{d(d+1)}{2}$, see \cite[Chapter II, Theorem 3.1]{Kobayashi1972}. 
\end{rema}

Before we go on with the special cases of the torus and the Euclidean space, let us be more specific about the initial datum for~\eqref{eq:transport}. First of all, the Cameron--Martin space of $\mu_{\alpha}$ is~\(H^\alpha (M)\).

If~\(M\) is compact, measure~\(\mu_\alpha\) may be realized by taking the pushforward measure of
\[
    \Phi_g^\alpha = \sum_{n\in \Nn} \lambda_n^{-\alpha} g_n e_n
\]
where~\((\lambda_n)_n\) are the ordered eigenvalues of~\(\sqrt{1-\Delta_g}\) counted with multiplicity, where~\((e_n)_n\) are the associated eigenfunctions, and where~\((g_n)_n\) is a family of independent real centered Gaussian variables. Since
\[
    \# \{ n \in \Nn \;| \; \lambda_n \in [2^k, 2^{k+1}) \} \sim 2^{dk}
\]
as~\(k\) goes to~\(\infty\) (see~\cite[Theorem 1.1]{hormander1968spectral}), we get that~\(\Phi_g^\alpha\) is almost surely in~\( H^{\alpha - \frac{d}{2} -} (M)\) but almost surely not in~\(H^{\alpha - \frac{d}{2}}\). Hence, it is in~\(L^2(M)\).

For the torus in particular,~\(\Phi_g^\alpha\) may be written
\[
    x\longmapsto \frac{1}{(2\pi)^{d/2}} \sum_{n\in \Zz^d} (1 + |n|^2)^{-\frac{\alpha}{2}} \tilde g_n e^{in\cdot x}
\]
where~\((\tilde g_n)_{n\in \Zz^d}\) is a family of complex centered jointly Gaussian variables whose law is characterized by the covariances for all~\(n,m \in \Zz^d\),
\[
    \mathbb E(\tilde g_m \tilde g_n) = \delta_{n+m},
\]
which implies that~\(\tilde g_{-n} = \overline{\tilde g_n}\). 

In the case of the Euclidean space, we may realise~\(\mu_\alpha\) as the real part of a Wiener integral
\[
    \Phi_g^\alpha (x) = \mathrm{Re}\int (1+|\xi|^2)^{-\frac{\alpha}{2}} e^{i x\cdot \xi} \mathrm{d}W(\xi).
\]
We recall that computationally speaking~\(\mathrm{d}W(\xi)\) is an infinitesimal centered Gaussian variable and we have
\[
    \E\left(\overline{\mathrm{d}W(\eta)} \mathrm{d}W(\xi)\right) = \delta(\xi - \eta) \mathrm{d}\xi \mathrm{d}\eta. 
\]
The Kolmogorov continuity theorem, see~\cite[Theorem 4.19]{hairer2009introduction} ensures that the sample fields are~\(\mu_\alpha\) almost surely locally~\(\mathcal C^{\alpha - \frac{d}{2}-}\). In particular, this ensures that the sample fields are almost surely in~\(L^2_{\mathrm{loc}}\) such that with proper cutoffs,~\(\mu_\alpha\) is a Gaussian measure on a Hilbert space. 

\begin{rema}
The zero-Fourier mode is \emph{a priori} conserved by the flow of~\eqref{eq:transport}. When~\(M\) is compact, instead of basing~\(\mu_\alpha\) on~\(L^2(M)\), we may base it on
\[
    L^2(M)_* = \left\{ \sum_{n\neq 0} a_n e_n \; : \; \sum_{n\neq 0} |a_n|^2 < \infty \right\} . 
\]
In this case, random field~\(\Phi_g^\alpha\) becomes
\[
    \sum_{n\neq 0} \lambda_n^{-\alpha} g_n e_n
\]
and solution~\(\Omega(t)\) is a well-defined field with vanishing constant mode.
\end{rema}

\begin{coro}[Case of $\Rr^d$]\label{thm.euclidean}
Assume that~\((M,g)\) is the Euclidean space $\Rr^d$ equipped with the flat metric. Then the following statements are equivalent: 
\begin{enumerate}[(i)]
    \item Gaussian measure $\mu_\alpha$ is quasi-invariant under the flow of \eqref{eq:transport}; 
    \item Gaussian measure $\mu_{\alpha}$ is invariant under the flow of \eqref{eq:transport}, that is, for all $t \in \mathbb{R}$ there holds $\mu_\alpha^t = \mu_{\alpha}$;
    \item vector field $u$ is of the form $u(t,x)=a(t) x + b(t)$ for some $a \in L^1_{\rm loc}(\mathbb{R},A_d(\mathbb{R}))$, where~\(A_d(\Rr)\) is the set of skew-symmetric matrices, and $b \in L^1_{\rm loc}(\mathbb{R},\mathbb{R}^d)$.
\end{enumerate}
\end{coro}

\begin{rema}\label{rema-isom}
The third item is equivalent to the fact that the characteristics~\(X_t\) and thus~\(X_t^{-1}\) are of the form~\(X_t^{-1}(x) = b(t) + R(t) x\) with~\(b(t) \in \Rr^d\) and~\(R(t) \in O_d(\Rr)\). Indeed, we have that 
\[
    \partial_t \Jac X_t  = \Jac u (t,X_t)\Jac X_t.
\]
We deduce that if~\(\Jac X_t \in O_d(\Rr)\) then~\(\Jac u(t,\cdot)\) is in the Lie algebra of~\(O_d(\Rr)\), which is the set of skew-symmetric matrices. Therefore,~\(\Jac u(t,x) = a(t)\) where~\(a\in A_d(\Rr)\) and~\(u(t,x) = a(t) x + c(t)\). Conversely, if 
\[
    u(t,x) = a(t) x + c(t)
\]
then~\(\partial_t \Jac X_t = a(t) \Jac X_t\) and thus
\[
    \partial_t (\Jac X_t^* \Jac X_t) = \Jac X_t^*(a(t)^* + a(t))\Jac X_t = 0.
\]
Since~\(\Jac X_{t=0} = \mathrm{Id}\), there exists~\(R(t) \in O_d(\Rr)\) such that~\(\Jac X_t = R(t)^{-1}\) and thus
\[
    X_t^{-1} = R(t) x + b(t).
\]
\end{rema}

\begin{coro}[Case of the torus]\label{thm.main}
Assume that~\((M,g)\) is the torus with the flat metric. Then the following statements are equivalent:
\begin{enumerate}[(i)]
    \item Gaussian measure $\mu_\alpha$ is quasi-invariant under the flow of \eqref{eq:transport};
    \item Gaussian measure $\mu_\alpha$ is invariant under the flow of \eqref{eq:transport};
    \item vector field $u$ is spatially constant, i.e. $u(t,x)=c(t)$, for a almost every $t$, and where $c$ is a given $L^1_{\rm loc}$ function.
\end{enumerate}
\end{coro}

\begin{rema} In these corollaries, we have made explicit the Killing vector fields on the torus and on the Euclidean space; explicit computations reveal that the Killing vector fields of the sphere $\mathbb{S}^{d-1}$ endowed with the standard metric are the infinitesimal generators of rotations, i.e. $u(x)=Ax$ for $A \in \mathcal{M}_d(\Rr)$, $A+A^T=0$ and $x \in \mathbb{S}^{d-1}$.
\end{rema}

\subsection{Sketch of the proof and further remarks}

Our proof relies heavily on the Feldman--H\'ajek theorem, which we recall as~\cref{th:FH}. This theorem characterizes the mutual absolute continuity between two Gaussian measures. In our case, this reduces to studying operator 
\[
\mathcal Q := (1-\Delta_g)^{\frac{\alpha}{2}} S(t) (1-\Delta_g)^{-\alpha} S(t)^* (1-\Delta_g)^{\frac{\alpha}{2}} - 1
\]
where the adjoint is taken in~\(L^2(M)\). Once we have applied the Feldman--H\'ajek theorem, we get that~\(\mu_\alpha^t\) and~\(\mu_\alpha\) are mutually absolutely continuous if and only if localisations of~\(\mathcal Q\) are Hilbert--Schmidt operators on~\(L^2(M)\), see~\cref{prop:criterion_for_abscont}. 

As in the PhD thesis of Knezevitch~\cite{knezevitch}, we then use pseudo-differential calculus. The key point is to prove that~\(\mathcal Q\) is an operator of order~\(0\) and to compute its principal symbol~\(\mathfrak q\). Using pseudo-differential calculus, we get that~\(\mathfrak q(x,\xi)\) depends only on~\(x\) and~\(\frac{\xi}{|\xi|}\) and that
\[
    \mathcal Q - \mathrm{Op}(\mathfrak q)
\]
is an operator of negative order. For this, we require that the regularity of~\(u\) is strictly bigger than~\(1\). More information about pseudo-differential calculus may be found in~\cref{sec.pseudo}. At this stage, we get that if~\(\mathcal Q\) is Hilbert--Schmidt, then~\(\mathfrak q\) vanishes. With the explicit computation of~\(\mathfrak q\), see~\eqref{eq.principal-symbol}, we get that~\(u\) is a Killing vector field.

Conversely, if~\(u\) is a Killing field, then~\(X_t\) is an isometry and thus~\(S(t)\) commutes with the Laplace--Beltrami operator, see \cite{Canzani2013}. We get that~\(\mathcal Q = 0\) and thus~\(\mu_\alpha^t = \mu_\alpha\).

\begin{rema}
When~\(\alpha=0\), the covariance operator \((1-\Delta_g)^{-\alpha}\) is the identity. Since $u$ is divergence-free, the transport flow preserves volume, and therefore $S(t)$ is unitary on $L^2(M)$. It follows that the Gaussian measure $\mu_0$ with identity covariance, namely the white-noise measure, is invariant under the flow of~\eqref{eq:transport}.
In the Euclidean case, this is the usual white noise on~$\Rr^d$; on~$\T^d$ it is the periodic white noise. 
This invariance is natural as white noise is invariant under volume-preserving transformations.
\end{rema}

\subsection{Questioning the methodology} 

Remark that the whole proof is at the level of the symbols of covariance operators and not at the level of random variables. This allows one to deal with generic Riemannian manifolds. Nevertheless, notice that even in the case of the torus, for which the eigenvalues and eigenvectors of the Laplace--Beltrami operator are easy to manipulate, explicit computations at the level of random variables are not trivial. We explain why.

Assume that initial data are taken as
\[
    \Omega_0^{\omega}(x)= \Phi_g^\alpha = \sum_{n\in\mathbb{Z}^d} \tilde g_n^{\omega}a_ne^{in\cdot x}, 
\]
with~\(a_n = (1+|n|^2)^{-\frac{\alpha}{2}}\).

Observe that $\mu^t_\alpha$ is the Gaussian measure characterized by the explicit formula:
\[
    \Omega^\omega(t)=\sum_{n\in\mathbb{Z}^d} g_n^{\omega}a_ne^{in\cdot X_t^{-1}(x)}.
\]
We can write 
\[
    \mathbb{E}\left[\overline{\hat{\Omega}(t,m)}\hat{\Omega}(t,m')\right] = \sum_{n\in \mathbb{Z}^d} |a_n|^2 \bar{I}_{n,m}I_{n,m'} \quad \mathrm{with}\quad I_{n,m}=\frac{1}{(2\pi)^d}\int_{\mathbb{T}^d}e^{in\cdot X_t^{-1}(x)-im\cdot x}\mathrm{d}x.
\]
Note that $I_{n,m}(t=0)=\delta_n^m$.

After applying the Feldman--H\'ajek theorem, the question of quasi-invariance amounts to the convergence of the sum
\begin{equation}
    \label{eq.hilb-schmidt-criteron}
    \sum_{m,m' \in \mathbb{Z}^d} \left(\frac{1}{|a_m||a_{m'}|}\mathbb{E}\left[\overline{\hat{\Omega}(t,m)}\hat{\Omega}(t,m')\right]-\delta_m^{m'}\right)^2.
\end{equation}
Consider the subsum on the diagonal~\(m=m'\)
\[
    S :=\sum_{m \in \mathbb{Z}_0^d} \left(\frac{1}{|a_m|^2}\mathbb{E}[|\hat{\Omega}(t,m)|^2]-1\right)^2.
\]
By substituting the formula for~\(\mathbb{E}[|\hat{\Omega}(t,m)|^2]\), we get
\[
    S = \sum_{m,n \in \mathbb{Z}_0^d} \left(\frac{1}{|a_m|^2}|a_n|^2 |I_{n,m}|^2-\delta_n^m\right)^2. 
\]
Not knowing the explicit form of~\(X_t\), in order to estimate~\(I_{n,m}\) we might resort to integration by parts. Integrating~\(e^{-im \cdot x}\) may gain powers of~\(m\) up to getting a converging sum in~\(m\) (to compensate~\(\frac{1}{|a_m|^2}\)) but it loses powers of~\(n\). This strategy thus fails. At this point, the standard change of variable to study
\[
    |I_{n,m}|^2 = \frac{1}{(2\pi)^{2d}} \int_{\T^d \times \T^d} e^{in\cdot (X_t x - X_t y) - im\cdot (x-y)}\mathrm{d}x\mathrm{d}y
\]
is to write
\[
    X_t (x) - X_t (y) = \int_0^1  D X_t(y + s(x-y)) \mathrm{d}s (x- y)
\]
but this is equivalent to applying pseudo-differential calculus, as may be seen in~\cref{lem:compo-symbole}.

\subsection{Notation}
In the rest of this paper, we write~\(P = (1-\Delta_g)^{-\alpha}\). 

Let~\(v : M \rightarrow M\) be a~\(\mathcal C^1\) function. We denote by $\Jac v$ the Jacobian matrix of $v$, that is the map such that for all~\(x\in M\), we have that~\(\Jac v(x)\) is a linear map from~\(T_xM\) to itself that matches in coordinates the $d\times d$ matrix with coefficients $[\Jac v(x)]_{ij}=\partial_j v_i(x)$.

Let $\mathfrak{S}_2(H)$ denote the space of Hilbert--Schmidt operators on a Hilbert space $H$ and $\mathfrak{S}_1(H)$ stands for the space of trace class operators on $H$.

The space~\(L^2(M)\) is the set of measurable functions on~\(M\) that are square-integrable endowed with the norm
\[
    \|f\|_{L^2(M)} = \left(\int_M |f|^2 \dVolg \right)^{\mez}.
\]
where $\dVolg(x) = \sqrt{\det g(x)}\,\mathrm{d}x^1 \cdots \mathrm{d}x^d$.
The space~\(H^s(M)\) is the set~\((1-\Delta_g)^{-s/2}L^2(M)\) endowed with the norm
\[
\|f\|_{H^s(M)} = \left( \int |(1-\Delta_g)^{\frac{s}{2}}f|^2  \dVolg \right)^{\mez} .
\]

If~\(A\) is a~\(d\times d\) matrix, then~\(A^\top\) is its transpose. Set $O_d(\mathbb{R})$ denotes the group of orthogonal matrices of size $d$, and $A_d(\mathbb{R})$ the space of skew-symmetric matrices of size $d$.

If~\(S\) is a linear map from a Hilbert space to a Hilbert space, we write~\(S^*\) its adjoint. If there is an ambiguity in the text,~\(S^*\) is the adjoint in~\(L^2(M)\).

For all functions~\(\theta : M \rightarrow \Rr\) or~\(\rho : \Rr^d \rightarrow \Rr\) we write~\(\supp \theta\) (resp.~\(\supp \rho\)) the support of~\(\theta\) (resp.~\(\rho\)).

Set~\(\mathcal C^\infty_c(M)\) denotes smooth functions with compact supports.

\subsection*{Organization of the paper}

The rest of the paper is organised as follows. In \cref{sec.FH} we reformulate quasi-invariance as the belonging to the Hilbert--Schmidt of a certain operator. In \cref{sec.proof} we are able to decompose this operator into several simple operators for which we can compute their Hilbert--Schmidt norms. This relies on pseudo-differential tools developed in \cref{sec.pseudo}.

\section{Absolute continuity between Gaussian measures}
\label{sec.FH}

The goal of this section is to prove the following characterization of quasi-invariance of $\mu_{\alpha}$ under the flow of \eqref{eq:transport}.

\begin{prop}\label{prop:criterion_for_abscont} Assume that~\(\mu_\alpha\) is quasi-invariant under the flow of~\eqref{eq:transport}. Then, for all~\(t,t_0 \in \Rr\) and all~\(\theta \in \mathcal C_c^\infty(M)\), we have that 
\[
   \theta P^{-\mez} [S(t,t_0), P] S(t,t_0)^{-1} P^{-\mez} \theta
\]
is a Hilbert--Schmidt operator on~\(L^2(M)\). 
\end{prop}

Before we go on, let us be more precise about the assumption on~\(\mu_\alpha\).

Let~\(x_0\in M\). For all~\(i\in \Nn^*\), let~\(B_i\) be the geodesic closed ball of center~\(x_0\) and radius~\(i\). Since~\(M\) is complete,~\(B_i\) is compact. For all~\(i\), let~\(\Phi_i\) be a smooth function compactly supported in~\(B_{i+1}\), such that~\(\Phi_i\) is equal to~\(1\) on~\(B_i\). Write~\(K_i\) for the support of~\(\Phi_i\). We see~\(L^2_{\mathrm{loc}}(M)\) as the projective limit of~\((E_i = L^2(K_i))_i\) with maps defined for~\(i<j\) as
\[
    \varphi_i^j : \begin{array}{rcl}
    E_j &\longrightarrow &E_i \\
    f &\longmapsto &(\Phi_i f)_{|K_i}.
    \end{array}
\]
We have indeed that~\(\varphi_i^j\) is continuous and that for~\(i<j<k\),
\[
    \varphi_i^j\circ \varphi_j^k = \varphi_i^k
\]
since by hypothesis~\(\Phi_i \Phi_j = \Phi_i\). We also set
\[
    \varphi_i : \begin{array}{rcl}
    L^2_{\mathrm{loc}} &\longrightarrow &E_i \\
    f &\longmapsto &(\Phi_i f)_{|K_i}.
\end{array}
\]
We also have that~\(\varphi_i\) is continuous, and the relations \(\varphi_i^j\circ \varphi_j = \varphi_i\).

The adjoint of $\varphi_i$ satisfies that for all $h \in E_i$, $\varphi_i^*h$ is the function in $L^2_{\rm loc}(M)$ defined as the extension by $0$ of $\Phi_ih$ outside $K_i$.

\begin{assumption}\label{As.FH}
We assume that for all~\(i\in \Nn^*\), measure~\(\mu_\alpha\) is such that~\((\varphi_i)_* \mu_\alpha\) is the centered Gaussian measure on~\(E_i\) with covariance operator~\(\varphi_i P \varphi_i^*\). 
\end{assumption}

\begin{rema}
We remark that this is coherent. First, by application of~\cref{lemma-pseudodiff-power}~(\textit{ii}), we have that~\(\Phi_i P \Phi_i\) is trace-class on~\(L^2(M)\), namely~\(\varphi_i P \varphi_i^*\) is trace-class on~\(E_i\). Then, we remark that
\[
    \varphi_i^j \varphi_j P \varphi_j^* (\varphi_i^j)^* = \varphi_i P \varphi_i^*.
\]
Finally, for all~\(L^2(M)\) functions~\(f,g\) compactly supported in~\(B_i\), we have that
\[
  \E_{\mu_\alpha}(\an{f,u}\an{u,g}) =  \E_{\mu_\alpha}(\an{f,\Phi_i u}\an{\Phi_i u,g}) = \E_{\mu_\alpha}(\an{f_{| K_i},\varphi_i (u)}\an{\varphi_i (u),g_{|K_i}})
\]
and then 
\[
    \E_{\mu_\alpha}(\an{f,u}\an{u,g}) = \E_{(\varphi_i)_*\mu_\alpha}(\an{f_{|K_i},u}\an{u,g_{|K_i}}) = \an{f, \Phi_i P \Phi_i g} = \an{f,Pg}.
\]
\end{rema}

The proof of~\cref{prop:criterion_for_abscont} is a consequence of the general criterion for absolute continuity of Gaussian measures of Feldman--H\'ajek, which we recall below.

\begin{theo}[Feldman--H\'ajek]\label{th:FH}\cite[Theorem 2.25]{DaPratoZabczyk2014}\\
Let~\(\nu_1\) and~\(\nu_2\) be two centered Gaussian measures over a Hilbert space $H$ with covariance operators~\(Q_1\) and~\(Q_2\). Then $\nu_1$ and $\nu_2$ are either mutually singular or mutually absolutely continuous with respect to each other. The latter happens if and only if 
\begin{enumerate}[(i)]
    \item measures~\(\nu_1\) and~\(\nu_2\) have the same Cameron--Martin space $Q_1^{\mez}(H)=Q_2^{\mez}(H)=:H_0$;
    \item $Q_1^{-\mez}Q_2^{\mez}(Q_1^{-\mez}Q_2^{\mez})^* - \mathbf{1} \in \mathfrak{S}_2(\bar H_0)$. 
\end{enumerate}
In the first item, $Q_i^{\mez}(H)$ is the range of the operator $Q_i^{\mez}$, and in the last item, $\bar H_0$ stands for the completion of $H_0$ under the $H$ norm, and the adjoint is understood with respect to $H$.
\end{theo}

Before we prove \cref{prop:criterion_for_abscont}, we recall a few elementary facts. 

\begin{lemm}\label{lem:invX}
For every~\(t,t_0\) we have~\(S(t,t_0)^* = S(t,t_0)^{-1}\), which is the composition with ~\(X_{t,t_0}\).
\end{lemm}

\begin{proof} 
By assumption, $X_{t,t_0}(\cdot)$ is the flow of a divergence-free vector field on~$(M,g)$, which implies that for any integrable function~$F$ on~$M$, there holds 
\[
    \int_M F(X_{t,t_0}(x)) \dVolg(x) = \int_M F(y) \dVolg(y).
\] 
Take $f,g \in L^2(M)$, and perform the change of variable $y = X_{t,t_0}(x)$ to get
\[ 
    \langle S(t,t_0)f,g\rangle_{L^2} = \int _M f(X_{t,t_0}^{-1}(y)) \overline{g(y)} \dVolg(y) = \int_M f(x) \overline{g(X_{t,t_0}(x))} \dVolg(x),
\]
which proves that $S(t,t_0)^*g = g \circ X_{t,t_0}$ and thus $S(t,t_0)^* = S(t,t_0)^{-1}$.
\end{proof}

\begin{lemm}\label{lem:snd-cov}
The measure~\((\varphi_i)_*\mu_{\alpha}^t\) has covariance~\(\Phi_i S(t)P S(t)^*\Phi_i\).
\end{lemm}

\begin{proof}
Note that $\mu_{\alpha}^t=S(t)_*\mu_\alpha$. Then, for all compactly supported test functions~\(f,g\), we have that~\(S(t)^*f\) and~\(S(t)^*g\) are compactly supported and thus
\begin{align*}
    \mathbb E\left[\overline{\langle S(t)\Omega_0,f\rangle_{L^2}}\langle S(t)\Omega_0,g\rangle_{L^2} \right] &= \mathbb E\left[\overline{\langle \Omega_0,S(t)^*f\rangle_{L^2}}\langle \Omega_0,S(t)^*g\rangle_{L^2} \right] \\ 
    &= \langle PS(t)^*f,S(t)^*g\rangle_{L^2} = \langle S(t)PS(t)^*f,g\rangle_{L^2}.
\end{align*}
We now use that for all~\(f,g\in E_i\), maps~\(\varphi_i^* f\) and~\(\varphi_i^* g\) are compactly supported, we get
\begin{multline*}
    \E_{(\varphi_i)_*\mu_{\alpha}^t} (\an{f,u}\an{u,g}) = \E_{\mu_\alpha^t}(\an{f, \varphi_i (u)}\an{\varphi_i(u), g}) = \an{\varphi_i^* f, S(t) P S(t)^* \varphi_i^* g} \\
    = \an{f,\Phi_i S(t) P S(t)^* \Phi_i g}.\qedhere
\end{multline*}
\end{proof}

\begin{proof}[Proof of \cref{prop:criterion_for_abscont}] 

Let~\(t,t_0 \in \Rr\). Because~\(\mu_\alpha\) is quasi-invariant under the flow of~\eqref{eq:transport} we have that~\(\mu_\alpha\) is equivalent to~\(\mu_\alpha^t\) and to~\(\mu_\alpha^{t_0}\). Let~\(\mu_\alpha^{t,t_0} = S(t,t_0)_* \mu_\alpha\). We prove that~\(\mu_\alpha\) and~\(\mu_\alpha^{t,t_0}\) are equivalent. Let~\(A \subset L^2_{\mathrm{loc}}\) be a measurable set. We have that
\[
\mu_\alpha^{t,t_0}(A) = 0 \quad \Leftrightarrow \mu_\alpha (S(t,t_0)^{-1}A) = 0.
\]
By equivalence of~\(\mu_\alpha\) and~\(\mu_\alpha^{t_0}\), we deduce
\[
\mu_\alpha^{t,t_0}(A) = 0 \quad \Leftrightarrow \quad \mu_\alpha^{t_0} (S(t,t_0)^{-1}A) = 0 \quad \Leftrightarrow \quad \mu_\alpha (S(t_0)^{-1}S(t,t_0)^{-1}A) = 0.
\]
We use that~\(S(t) = S(t,t_0) S(t_0)\) to get
\[
\mu_\alpha^{t,t_0}(A) = 0 \quad \Leftrightarrow \quad \mu_\alpha (S(t)^{-1}A) = 0 \quad \Leftrightarrow \quad \mu_\alpha^t (A) = 0.
\]
We use the equivalence of~\(\mu_\alpha^t\) and~\(\mu_\alpha\) to get that for all measurable set~\(A\) in~\(L^2_{\mathrm{loc}}\), we have
\[
    \mu_\alpha^{t,t_0}(A) = 0 \Longleftrightarrow \mu_\alpha(A) = 0.
\]
We recall that~\(\varphi_i\) is continuous such that, for all measurable set~\(A\) in~\(E_i\), we have
\[
    \mu_\alpha^{t,t_0}(\varphi_i^{-1}(A)) = 0 \Longleftrightarrow \mu_\alpha(\varphi_i^{-1}(A)) = 0,
\]
that is to say
\[
    (\varphi_i)_*\mu_\alpha^{t,t_0}(A) = 0 \Longleftrightarrow (\varphi_i)_*\mu_\alpha(A) = 0.
\]
The covariance of~\((\varphi_i)_*\mu_\alpha^{t,t_0}\) is given by
\[
    P_i^t = \varphi_i S(t,t_0) P S(t,t_0)^{-1} \varphi_i^*,
\]
and we set~\(P_i = P_i^{t_0}=\varphi_iP\varphi_i^*\).

The Cameron--Martin space of~\((\varphi_i)_*\mu_\alpha\) is the completion of
\[
    H_i = \{ f\in E_i \; : \; \exists f^*\in E_i,\; \forall g \in E_i, \; \an{\varphi_i P\varphi_i^* f^* , g}_{E_i} = \an{f,g}\}
\]
for the norm
\[
    \|f\|_i^2 = \an{f^*,\varphi_i P\varphi_i^* f^*},
\]
that is to say
\[
    H_i = \mathrm{range}(\varphi_i P \varphi_i^*)
\]
completed by the norm
\[
    \|f\|_i^2 = \an{f, (\varphi_i P \varphi_i^*)^{-1} f}_{E_i} .
\]
In other words, the Cameron--Martin space of~\((\varphi_i)_*\mu_\alpha\) and thus of~\((\varphi_i)_*\mu_\alpha^{t,t_0}\) is~\(\varphi_i H^\alpha(M)\).

Set~\(\tilde P_i = \Phi_i P \Phi_i\) and~\(\tilde H_i\) its range. We have that~\(\tilde{H}_i = \iota_i^* H_i\) where~\(\iota_i : f \mapsto f_{|K_i}\).

Let~\((e_n)_n\) be a basis of~\(L^2(M)\) of eigenfunctions of~\(\tilde P_i\) and~\(\lambda_n \geq 0\) the associated eigenvalues. Let~\(N\) be the set of integers~\(n\) such that~\(\lambda_n\neq 0\). We have that~\((e_n)_{n\in N}\) form a Hilbert basis of~\(\overline{\tilde H_i}\) --- the closure is taken in~\(L^2(M)\). For all~\(n\in N\), set~\(f_n = \iota_i e_n \in E_i\). Since~\(e_n\) is supported in~\(K_i\), we have
\[
    \an{f_n, f_m} = \an{e_n,e_m} = \delta_n^m, \quad P_i f_n = \iota_i \tilde P_i e_n = \lambda_n f_n.
\]
And since~\(e_n \in \tilde H_i\), we have that~\(f_n \in  H_i\).

We have that operator~\(P_i^{-\mez} P_i^t P_i^{-\mez} -1 = P_i^{-\mez}(P_i^t - P_i) P_i^{-\mez} \) is Hilbert--Schmidt in~\(\overline{H_i}\) --- the closure being taken in $E_i$. Since~\((f_n)_{n\in N}\) form an orthonormal family in~\(\overline{H_i}\), we deduce that
\[
    S : = \sum_{n,m\in N} |\an{f_n, P_i^{-\mez}(P_i^t - P_i) P_i^{-\mez} f_m}_{E_i}|^2 < \infty ,
\]
that is
\[
    S=\sum_{n,m\in N} \frac1{\lambda_n \lambda_m }|\an{f_n, (\varphi_i S(t,t_0) PS(t,t_0)^{-1}\varphi_i^* - P_i) f_m}_{E_i}|^2 <\infty.
\]
Since~\(f_n = \iota_i e_n\), we deduce
\[
    S=\sum_{n,m\in N} \frac1{\lambda_n \lambda_m }|\an{e_n,\Phi_i (S(t,t_0) PS(t,t_0)^{-1} - P) \Phi_i e_m}|^2 <\infty.
\]
Set~\(Q_i = \Phi_i (S(t,t_0) PS(t,t_0)^{-1} - P) \Phi_i\). Recall that by hypothesis the Cameron--Martin space of~\(\mu_\alpha^{t,t_0}\) is the same as the one of~\(\mu_\alpha\) which implies that
\[
    \overline{\mathrm{range}(\Phi_i S(t,t_0) PS(t,t_0)^{-1}  \Phi_i)} = \overline{\tilde{H}_i}.
\]
We deduce
\begin{align*}
    S = \sum_{n,m\in N} \frac1{\lambda_n \lambda_m }|\an{e_n,Q_i e_m}|^2 & = \sum_{n,m\in N} \an{Q_i \tilde P_i^{-1 }e_m,e_n}\an{\tilde P_i^{-1} e_n,Q_i e_m} \\
    & = \sum_{n,m\in N} \an{Q_i \tilde P_i^{-1 }e_m,e_n}\an{ e_n,\tilde P_i^{-1} Q_i e_m} \\
    & = \sum_{m\in N} \an{Q_i \tilde P_i^{-1 }e_m,\tilde P_i^{-1} Q_i e_m},
\end{align*}
where~\(\tilde P_i^{-1}\) is defined from~\(\overline{\tilde H_i}\) to itself. Note that we have used that~\((e_n)_{n\in N}\) is a Hilbert basis of~\(\overline{\tilde H_i}\).
In other words,
\[
    S = \mathrm{Tr} \left( Q_i \tilde P_i^{-1} Q_i \tilde P_i^{-1}  \right)
\]
where the trace is taken in~\(\overline{\widetilde H_i}\).

Write~\(\tilde P_i = \Phi_i P^\mez  P^\mez \Phi_i =(P^\mez \Phi_i)^*(P^\mez \Phi_i)\) so that on~\(\tilde H_i\)
\[ 
    \tilde P_i^{-1} = (P^\mez \Phi_i)^{-1}((P^\mez \Phi_i)^{*})^{-1} = (P^\mez \Phi_i)^{-1}(\Phi_i P^\mez)^{-1}.
\]
By cyclicity of the trace we get that 
\[
    (\Phi_i P^\mez )^{-1} Q_i (P^\mez \Phi_i)^{-1}
\]
is Hilbert--Schmidt on~\( \overline{(\Phi_i P^\mez )^{-1} \tilde H_i}\). Let $P(t)=S(t,t_0)PS(t,t_0)^{-1}$ and observe that 
\[
    Q_i = (P^\mez \Phi_i)^*P^{-\mez}(P(t)-P)P^{-\mez}(P^\mez \Phi_i),
\]
so that \(P^{-\mez}(P(t)-P)P^{-\mez}\) is Hilbert--Schmidt on~\(F_i = \overline{\mathrm{range}(P^\mez \Phi_i)}\). In other words, if \(\Pi_i\) denotes the orthogonal projection of \(L^2(M)\) on \(F_i\), then 
\[
    \Pi_i P^{-\mez} (P(t) - P) P^{-\mez} \Pi_i \in \mathfrak{S}_2(L^2(M)).
\]

Let~\(\theta\) be compactly supported and let~\(i\in \Nn^*\) such that the support of~\(\theta\) is included in~\( B_{i-1}\). If~\(P^{-\mez}\) is local, we automatically get that the composition of~\(\Pi_i\) and the multiplication by~\(\theta\) is equal to the multiplication by~\(\theta\). In general, this is not the case; we make explicit the defect between those two operators.

Let~\(\theta_0 \in \mathcal C_c^\infty(\ring B_i)\) be such that $\theta_0 = 1$ on a neighborhood of $B_{i-1}$. Let
\[
    v \in  \overline{\mathrm{range}(P^{\mez}\Phi_i)}^{\perp} = \ker(P^{\mez}\Phi_i)^* = \ker(\Phi_i P^{\mez}). 
\]
Then $\Phi_i P^{\mez}v = 0$ and because $\theta_0 \Phi_i = \theta_0$ we have $\theta_0 P^{\mez}v=0$. It follows that $P^{\mez}v = (1-\theta_0)P^{\mez}v$ so that 
\[
    \theta v = \theta P^{-\mez}(1-\theta_0)P^{\mez}v. 
\]
In other words,
\[
    \theta(1- \Pi_i) = \theta P^{-\mez}(1-\theta_0)P^{\mez} (1-\Pi_i).
\]
We deduce in particular thanks to~\cref{as} that~\(\theta(1-\Pi_i)\) is Hilbert--Schmidt.

Write~\(Q(t) = P^{-\mez}(P(t) - P) P^{-\mez}\). We have that
\begin{align*}
\theta Q(t) \theta = & \theta \Pi_i Q(t) \Pi_i \theta \\
& + \theta (1-\Pi_i) Q(t) \Pi_i \theta \\
& + \theta \Pi_i Q(t) (1-\Pi_i) \theta \\
& + \theta (1-\Pi_i) Q(t) (1-\Pi_i ) \theta
\end{align*}
Because the multiplication by~\(\theta\) is bounded, we get that~\(\theta \Pi_i Q(t) \Pi_i \theta \) is Hilbert--Schmidt. Since~\(u \in L^1_{\mathrm{loc}}(\mathbb{R},\mathcal C^\alpha(M))\) with $\alpha > \frac{d}{2}$, we get that~\(X_{t,t_0} \in \mathcal C^\alpha (M)\), we deduce that~\(S(t,t_0)\) is a bounded operator from~\(H^\alpha(M)\) to~\(H^\alpha (M)\) and thus that~\(Q(t)\) is bounded. Since~\(\theta(1-\Pi_i)\) and~\((1-\Pi_i)\theta=(\theta(1-\Pi_i))^*\) are Hilbert--Schmidt, we deduce that~\(\theta Q(t) \theta\) is Hilbert--Schmidt.
\end{proof}

\section{Pseudo-differential calculus}
\label{sec.pseudo} 

In view of \cref{prop:criterion_for_abscont}, the quasi-invariance of the Gaussian measure $\mu_{\alpha}$ is now reduced to understanding whether the operator
\[
    \theta P^{-\mez}(S(t,t_0)PS(t,t_0)^{-1}-P)P^{-\mez}\theta ,\qquad \theta \in C^{\infty}_c(M), 
\]
is Hilbert--Schmidt on $L^2(M)$ or not.

At this point, we rely on pseudo-differential calculus. This section collects results from pseudo-differential calculus that will be needed in the sequel, and refers the reader to the appropriate references for the proofs of the statements. 

\subsection{Pseudo-differential operators on \texorpdfstring{$\Rr^d$}{Rd}}

In this article, we use pseudo-differential calculus only in local coordinates on $\mathbb{R}^d$, in order not to add difficulties that may arise on manifolds, especially in the non-compact case. We refer to \cite{hormanderIII} for a systematic treatment, and to \cite{Fermanian2014} for a discussion of some of the difficulties of pseudo-differential calculus on manifolds.

\begin{defi} 
Let $m \in \mathbb{R}$. A function $a : \mathbb{R}^d \times \mathbb{R}^d \to \mathbb{C}$ is said to belong to $S^m$ if $a$ is smooth and satisfies for all multi-indices $\alpha, \beta$: 
\[
    |\partial_{\xi}^{\alpha}\partial_x^{\beta}a(x,\xi)| \leq C(\alpha,\beta) \langle \xi \rangle^{m-|\alpha|}.
\]
Similarly, we say that a function $a : \mathbb{R}^d \times \mathbb{R}^d \to \mathbb{R}$  belongs to $C_*^{\rho}S^m$ if $a(x,\xi)$ is smooth in $\xi$ and for all multi-indices $\alpha$:
\[
    \|\partial_{\xi}^\alpha a(\cdot, \xi)\|_{C_*^{\rho}} \leq C(\alpha) \langle \xi \rangle^{m-|\alpha|},
\] 
where $C_*^{\rho}$ denotes the Zygmund space \cite[Appendix A]{TaylorPDE}.
\end{defi}

\begin{defi}[Classical symbols]\label{def.classical-symbol}
Let $m \in \mathbb{R}$. A function $p : \mathbb{R}^d \times \mathbb{R}^d \to \mathbb{R}$ is a classical symbol of order $m$, and we write $p \in S_{\rm cl}^m$, if $p\in S^m$ and if for all~\(j\in \Nn\) there exists $p_{m-j}\in S^{m-j}$, homogeneous of degree $m-j$ in $\xi$, smooth for $\xi \neq 0$ such that 
\[
    p - \sum_{j=0}^N\omega(\xi)p_{m-j} \in S^{m-N-1}\quad \text{for all } N\geq 0,
\] 
where $\omega \in \mathcal{C}^{\infty}(\mathbb{R}^d)$ vanishes in a neighborhood of $0$ and equals $1$ for $|\xi| \geq 1$. 
We write $\sigma_m(p)=p_m$ and we say that $p_m$ is the principal symbol of $p$.
\end{defi}

We will write $\Op(a)$ to denote the Kohn--Nirenberg quantization of a symbol $a \in S^m$, namely for any $u\in \mathcal{S}(\mathbb{R}^d)$,  
\begin{equation}
    \label{eq.KN-quantization}
    \Op(a)u(x)=\int_{\Rr^d} e^{ix\cdot \xi}a(x,\xi)\hat{u}(\xi)\mathrm{d}\xi = \iint_{\Rr^d \times \Rr^d} e^{i(x-y)\cdot \xi}a(x,\xi)u(y)\frac{\mathrm{d}y\mathrm{d}\xi}{(2\pi)^d}.
\end{equation}
In the above, $\hat{u}(\xi)$ denotes the Fourier transform of $u$. We say that $\Op(a)$ is a pseudo-differential operator with symbol $a$. Note that in particular the Schwartz kernel of $\Op(a)$ is given by 
\[
    K(x,y)= \int_{\mathbb{R}^d}e^{i(x-y)\cdot \xi} a(x,\xi)\frac{\mathrm{d}\xi}{(2\pi)^d},
\]
which should be understood in the sense of oscillatory integrals. From integration by parts we obtain that $K \in C^{\infty}(\mathbb{R}^{2d} \setminus \{(x,x) : x\in \mathbb{R}^d\})$.

One feature of pseudo-differential operators is the algebraic calculus they enjoy.

\begin{prop}[Composition of pseudo-differential operators] \cite[Théorème 10.13]{Alazard2023}\label{prop.compo.pseudo}
Let $m,n \in \mathbb{R}$ and $a \in S^m$ and $b\in S^n$. Then there exists $a\#b \in S^{m+n}$ satisfying
\[
    \Op(a) \circ \Op(b)=\Op(a\#b). 
\]
Moreover, $a\#b$ admits the following asymptotic expansion: 
\[
    a\#b (x,\xi) \sim \sum_{\alpha \in \mathbb{N}^d} \frac{1}{i^{|\alpha|}\alpha !}\partial_{\xi}^\alpha a(x,\xi) \partial_x^{\alpha}b(x,\xi),
\]
that is for all $N\geq 1$, the symbol $a\# b - \displaystyle \sum_{|\alpha|<N}\frac{1}{i^{|\alpha|}\alpha !}\partial_{\xi}^\alpha a(x,\xi) \partial_x^{\alpha}b(x,\xi)$ belongs to $S^{m+n-N}$. 
\end{prop}

\begin{rema} Note that the following exact identities hold: 
\[
    \Op(a(x,\xi))\Op(b(\xi))=\Op(a(x,\xi)b(\xi)),    
\]
\[
    \Op(a(x))\Op(b(x,\xi))=\Op(a(x)b(x,\xi)).
\]
\end{rema}

\begin{rema}[Limited regularity]
It follows from the proof in \cite[Théorème 10.13]{Alazard2023} that if $a \in C^{\rho}_*S^m$ and $b\in C^{\sigma}_*S^n$ with $\sigma > 1$ and $\rho >0$, then $\Op(a)\circ \Op(b) = \Op(ab) \mod \Op(C^\tau_*S^{m+n-1})$ with $\tau = \min\{\rho, \sigma - 1\}$.
\end{rema}

We will later rely on the following key symbolic calculus tool, which aims at understanding how conjugation with the composition by a diffeomorphism interacts with pseudo-differential operators. It is essentially Egorov's theorem \cite[Theorem 0.9.A]{TaylorPDE}, which in our setting is a localized version of \cite[Theorem 4.5]{AlazardShao2024}.

\begin{prop}\label{lem:compo-symbole}
Let $p \in S^m(\mathbb{R}^d\times\mathbb{R}^d)$. Let $U, V$ be two open sets of $\Rr^d$, $V$ being convex, and let $\chi : U \to V$ be a smooth diffeomorphism satisfying $\|\Jac \chi^{-1} - I_d\|_{L^{\infty}(V)} <1$. 
Let $\psi \in \mathcal{C}^{\infty}_c(V)$ and $\rho \in \mathcal{C}_c^{\infty}(U)$ 
such that~\(\rho\) is equal to~\(1\) on a neighborhood of $\chi^{-1}(\operatorname{supp} \psi)$. If~\(f\) is compactly supported in~\(V\) set~\(C_\chi f\) the function corresponding to~\(f\circ \chi\) on~\(U\) and~\(0\) outside~\(U\). Operator~\(C_\chi\) is a bijection between functions compactly supported in~\(V\) and functions compactly supported in~\(U\).
Set 
\[
Q = \psi C_{\chi}^{-1} \rho \Op(p) \rho C_{\chi}\psi,
\]
and define $J(x,y)=\displaystyle\int_0^1 \Jac \chi^{-1}(tx+(1-t)y)^\top\mathrm{d}t$ for $x, y \in V$.

Then, there exists $q\in S^m(\Rr^d \times \Rr^d)$ such that $Q=\psi\Op(q)\psi$, and 
moreover, $q$ admits the following asymptotic expansion on a neighborhood of $\supp \psi$ included in~\(V\):
\begin{equation}
    \label{eq.asymptotic-expansion}
        q(x,\xi) \sim \sum_{\alpha \in \mathbb{N}^d} \frac{\rho(\chi^{-1}(x))}{i^{|\alpha|}\alpha !} \partial_y^\alpha \left.\partial_{\xi}^\alpha \left(\rho(\chi^{-1}(y))p\left(\chi^{-1}(x),J(x,y)^{-1}\xi\right)\left\vert\frac{\det \Jac \chi^{-1} (y)}{\det J(x,y)}\right\vert\right)\right\vert_{y=x}.
\end{equation}
\end{prop}

\begin{rema}[Limited regularity]
In the case of a diffeomorphism with limited regularity $\chi \in C^{1+2\delta}_*$, $\delta \in (0,1)$, we would have
\[
    Q=\psi \Op\left(\rho^2(\chi^{-1}(x))p(\chi^{-1}(x),\Jac \chi^{-1}(x)^{-\top}\xi)\right)\psi \mod \Op\left(C^{\delta}_*S^{m-\delta}\right).
\]
\end{rema}

\begin{proof} 
Choose $\theta\in C_c^\infty(V)$ such that $\theta=1$ on a neighborhood of $\supp\psi$ and $\chi^{-1}(\supp\theta)$ is a compact subset of $U$. Define $\widetilde{Q}=\theta C_{\chi^{-1}}\rho\Op(p)\rho C_\chi\theta$. Since $\psi\theta=\psi$, we have $Q=\psi\widetilde Q\psi$. It is therefore enough to prove that $\widetilde{Q}=\Op(q)$ for some $q\in S^m$, and then to identify the expansion of $q$ on the region where $\theta=1$.

Recall that the Schwartz kernel of $P=\Op(p)$ is given in the oscillatory sense by
\[
      K_P(x,y) =(2\pi)^{-d}\int_{\Rr^d} e^{i(x-y)\cdot\zeta}p(x,\zeta)\mathrm{d}\zeta.
\]
For a test function $u$, the change of variables $z=\chi^{-1}(y)$ gives 
\begin{multline*}
    (\widetilde{Q}u)(x) =\theta(x)\rho(\chi^{-1}(x)) \int_U K_P(\chi^{-1}(x),z)\rho(z)\theta(\chi(z))u(\chi(z))\mathrm{d}z \\ 
=\theta(x)\rho(\chi^{-1}(x))\int_V K_P(\chi^{-1}(x),\chi^{-1}(y))\rho(\chi^{-1}(y))\theta(y)u(y)|\det \Jac\chi^{-1}(y)|\mathrm{d}y.
\end{multline*}
Therefore if we write $\rho_0(x)=\theta(x)\rho(\chi^{-1}(x))$ we have
\[
    K_{\tilde Q}(x,y)=(2\pi)^{-d}\rho_0(x)\rho_0(y)|\det \Jac\chi^{-1}(y)|\int_{\Rr^d} e^{i(\chi^{-1}(x)-\chi^{-1}(y))\cdot\zeta}p(\chi^{-1}(x),\zeta)\mathrm{d}\zeta.
\]
Choose $\Theta\in C_c^\infty(\Rr^d\times\Rr^d)$ such that
\begin{equation}
  \Theta(x,y)=1\quad\text{if }x,y\in\supp\rho_0\text{ and }|x-y|\leq\delta,
\end{equation}
for some fixed $\delta >0$. Write $K_{\widetilde{Q}}=K_{\text{diag}} + K_{\text{reg}} = \Theta K_{\widetilde{Q}} + (1-\Theta)K_{\widetilde{Q}}$ and similarly $\widetilde{Q}=\widetilde{Q}_{\text{diag}} + \widetilde{Q}_{\text{reg}}$. Since $K_P$ is smooth away from the diagonal, it follows that $K_{\text{reg}} \in C_c^{\infty}(\Rr^d\times\Rr^d)$. This implies that $\widetilde{Q}-\widetilde{Q}_{\text{diag}}$ is a regularizing operator. 

From the fundamental theorem of calculus we can write $(\chi^{-1}(x)-\chi^{-1}(y)) \cdot \zeta = (x-y)\cdot J(x,y)\zeta$. Moreover $\|J(x,y)-I_d\|_{L^{\infty}} \leq \|\Jac \chi^{-1} (x) - I_d\|_{L^{\infty}}<1$, which is enough to ensure that $J(x,y)$ is invertible, uniformly on compact sets with bounded derivatives. Performing the change of variables $\xi = J(x,y)\zeta$ yields
\[
    K_{\operatorname{diag}}(x,y)=(2\pi)^{-d}\int_{\Rr^d}e^{i(x-y)\cdot \xi}a(x,y,\xi)\mathrm{d}\xi
\]
with 
\[
    a(x,y,\xi)=\Theta(x,y)\rho_0(x)\rho_0(y)p(\chi^{-1}(x),J(x,y)^{-1}\xi)\frac{|\det \Jac \chi^{-1}(y)|}{|\det J(x,y)|}.
\]
From the chain rule we infer 
\begin{equation}
\label{eq.symbol-est}
     |\partial_x^{\alpha}\partial_y^{\beta}\partial_{\xi}^\gamma a(x,y,\xi)| \leq C(\alpha,\beta,\gamma)\langle \xi \rangle^{m-|\gamma|}.
\end{equation}
Therefore we can write $\tilde{Q}_{\text{diag}} = \Op(q_{\text{diag}})$ with 
\[
    q_{\text{diag}}(x,\xi)=(2\pi)^{-d} \iint e^{-iz\cdot \zeta}a(x,x+z,\xi+\zeta)\mathrm{d}z\mathrm{d}\zeta,
\]
which in view of \eqref{eq.symbol-est} belongs to $S^m(\Rr^d \times \Rr^d)$. 

We now resort to a Taylor expansion of $a$. Fix $N\geq1$ and write
\begin{align}
\label{eq.taylor}
    a(x,x+z,\xi+\zeta) =\sum_{|\alpha|<N} \frac{z^\alpha}{\alpha!}(\partial_y^\alpha a)(x,x,\xi+\zeta) +R_N(x,z,\xi+\zeta),
\end{align}
where the integral form of the remainder is
\begin{equation}
\label{eq.taylor-remainder}
    R_N(x,z,\xi + \zeta) =N\sum_{|\alpha|=N}\frac{z^\alpha}{\alpha!} \int_0^1(1-t)^{N-1} (\partial_y^\alpha a)(x,x+tz,\xi+\zeta)\mathrm{d}t.
\end{equation}
We can now leverage the identity 
\begin{equation}
\label{eq.identity-distrib}
    (2\pi)^{-d}\iint e^{-iz\cdot\zeta}z^\alpha F(\xi+\zeta)\mathrm{d} z\mathrm{d}\zeta =i^{-|\alpha|}\partial_\xi^\alpha F(\xi),
\end{equation}
and deduce
\[
    q_{\text{diag}}(x,\xi)=\sum_{|\alpha|<N} \frac{1}{i^{|\alpha|}\alpha!}\partial_\xi^\alpha\partial_y^{\alpha}a(x,y,\xi)\vert_{y=x} + r_N(x,\xi), 
\]
which is the required asymptotic series. To conclude the proof, it remains to observe that 
\[
    r_N(x,\xi)=  (2\pi)^{-d}\iint e^{-iz\cdot \zeta} R_N(x,z,\xi + \zeta) \mathrm{d} z \mathrm{d} \zeta
\]
satisfies
\[
    |\partial_x^{a}\partial_\xi^br_N(x,\xi)| \leq C(N,a,b)\langle \xi \rangle^{m-N-|b|}.
\]
Finally \eqref{eq.asymptotic-expansion} is obtained because $\theta = 1$ on a neighborhood of $\supp \psi$.
\end{proof}

\subsection{Pseudo-differentiality of the powers of \texorpdfstring{$\operatorname{id}-\Delta_g$}{I-Delta g}}

The goal of this section is to prove the following.

\begin{prop}\label{lemma-pseudodiff-power}
Let $(M,g)$ be a smooth complete boundaryless Riemannian manifold. Let $L=(\operatorname{id}-\Delta_g)$ and $\beta \in \mathbb{R}$.
Consider a local chart $\kappa : M \supset U \to V \subset \Rr^d$ and $\theta_1, \theta_2 \in \mathcal{C}^{\infty}_c(U)$. Let also~\(\rho_1,\rho_2 \in \mathcal C^\infty_c(M)\). Then: 
\begin{enumerate}[(i)]
    \item The operator\footnote{Extension to $\mathbb{R}^d$ by zero is assumed} $\kappa_*\theta_1 L^{\beta} \theta_2 \kappa^*$ is a classical pseudo-differential operator on $\Rr^d$ of order $2\beta$, whose principal symbol is
    \[
        a_{2\beta}(y,\eta)=\theta_1(\kappa^{-1}(y))\theta_2(\kappa^{-1}(y))(g^{ij}(\kappa^{-1}(y))\eta_i\eta_j)^{\beta}, \qquad y \in \Rr^d, \quad |\eta|\geq 1.
    \]
    \item Assume $\beta < - \frac{d}{2}$. Then $\rho_1L^{\beta}\rho_2 \in \mathfrak{S}_1(L^2(M))$.
    \item Assume that $\rho_2 = 0$ on a neighborhood of $\supp \rho_1$. Then the Schwartz kernel of $\rho_1 L^{\beta} \rho_2$ is compactly supported and smooth. In particular: 
    \[
        \rho_1 L^{\beta} \rho_2 \in \mathfrak{S}_1(L^2(M)) \subset \mathfrak{S}_2(L^2(M)).
    \]
    \item Assume that $\rho_2 = 1$ on a neighborhood of $\supp \rho_1$. Then: 
    \[
        \rho_1 L^{\beta}(1-\rho_2) \in \mathfrak{S}_2(L^2(M)),
    \]
    in the following cases: $M$ compact; or $M=\Rr^d$ endowed with the Euclidean metric; or if $\beta$ is a natural integer. 
\end{enumerate}
\end{prop}

Before we go on with the proof, let us be more precise about the definition of~\(L^\beta\). The spectrum of~\(L\) is included in~\([1,\infty)\). We define the following contour $\Gamma = \Gamma_+ \cup \Gamma_0 \cup \Gamma_-$ of~\([1,\infty)\): fix some \(r_0\in (0,1)\) and define
\begin{equation}
    \label{eq-Gamma1}
    \Gamma_+ = \left\{1+ r e^{i\omega_0}: r\in [r_0,+\infty) \right\}, \qquad \Gamma_- = \left\{1+ r e^{-i\omega_0} : r\in [r_0,+\infty)\right\},
\end{equation}
\begin{equation}
    \label{eq-Gamma2}
    \Gamma_0 = \left\{1+ r_0e^{i\omega} : \omega\in [\omega_0, 2\pi-\omega_0] \right\}.
\end{equation}

Define 
\[
    (k,\gamma)= \begin{cases} (0,-\beta) &\text{if } \beta <  0 \\ 
    (\lceil \beta \rceil, \lceil \beta \rceil - \beta) & \text{if } \beta >0,
    \end{cases}
\]
and write $L^{\beta}=L^kL^{-\gamma}$.

If~\(\gamma = 0\) then~\(L^\beta\) is a differential operator and the results in~\cref{lemma-pseudodiff-power} hold. We focus on the case~\(\gamma>0\).

Thanks to Cauchy's formula and functional calculus for $L$, which is essentially self-adjoint \cite{gaffney1951}, we define
\begin{equation}\label{eq:defLbeta}
    L^{\beta}=L^k \frac{i}{2\pi}\int_{\Gamma} \lambda^{-\gamma}(L-\lambda)^{-1}\mathrm{d}\lambda,
\end{equation}
where $\Gamma$ is oriented counter-clockwise and defined as in \eqref{eq-Gamma1} and \eqref{eq-Gamma2}. 
In the above $\lambda^{-\gamma}$ is defined with the principal branch of the logarithm in $\mathbb{C}\setminus (-\infty,0]$.

This formula is consistent with spectral powers of~\(L\) in the compact case. 

The proof of (\textit{i})~\cref{lemma-pseudodiff-power} is originally due to Seeley \cite{seeley1967} and a proof is given in \cite[Sections 9--11]{Shubin2001}. However, in these references, the ambient manifold is compact. The proof seems adaptable in non compact settings but relies on a pseudo-differential description for the resolvent of $L=-\Delta_g + \operatorname{id}$ that is valid on the whole manifold. It is then injected (with cutoffs) in~\eqref{eq:defLbeta}. In~\cite{Bouclet2012}, such a description is provided locally: it is called a parameter-dependent parametrix construction for the resolvent of $L=-\Delta_g + \operatorname{id}$. We indicate how (\textit{i})~\cref{lemma-pseudodiff-power} follows from this parameter-dependent parametrix construction. Remark that the main issue is to obtain uniform bounds in~\(\lambda\), which are provided in~\cite{Bouclet2012}.

It will be convenient for us to use the following notation: $c(\lambda; \cdot, \cdot) \in S^{m}_{\lambda}$ if for all $a, b \in \mathbb{N}^d$, 
\[
      \sup_{(y,\eta) \in \Rr^d \times \Rr^d}|\partial_{y}^a\partial_{\eta}^bc(\lambda ; y,\eta)| \leq C(a,b)(\langle \lambda \rangle^{\frac{1}{2}} + \langle \eta \rangle)^{m-|b|}.
\]
\begin{lemm}[Parameter dependent construction] \label{lemma-parametrix}
Consider a local chart $\kappa : M \supset U \to V \subset \Rr^d$ with~\(V\) open and $\rho_0, \rho_1, \rho_2 \in C^{\infty}_c(V)$ such that $\rho_{k+1} = 1$ on a neighborhood of $\supp \rho_k$, for $k\in\{0,1\}$. 
There exist $q_{-2-j}(\lambda ; \cdot, \cdot) \in S^{-2-j}_{\lambda}$, $j\geq 0$, and $r_N(\lambda ; \cdot, \cdot) \in S^{-N}_{\lambda}$, $N\geq 1$ such that for all $\lambda \in \Gamma$ we have
\begin{equation}
    \label{eq-bouclet}
   (\kappa_*(L-\lambda)\kappa^*)\rho_1\Op\left( \sum_{j=0}^{N-1}q_{-2-j}(\lambda)\right)\rho_0 = \rho_0 - \rho_2\Op(r_N(\lambda))\rho_0,\qquad  N \geq 1.
\end{equation}
Moreover, $\lambda \mapsto q_{-2}(\lambda)\in S^{-2}_\lambda$ is holomorphic, the functions $\lambda \mapsto q_{-2-j}(\lambda) \in S^{-2-j}_{\lambda}$ and $\lambda \mapsto r_N(\lambda) \in S^{-N}_{\lambda}$ are continuous, and the following decomposition holds: 
\begin{equation}
    \label{eq.ecriture-symbolej}
    q_{-2-j}(\lambda; y,\eta) = \sum_{\ell=1}^{2j}|\lambda - 1|^{\ell-\frac{j}{2}}\frac{d_{j\ell}(y,|\lambda-1|^{-\frac{1}{2}}\eta)}{(p_2(y,\eta) + 1 - \lambda)^{1+\ell}}, \qquad j \geq 1,
\end{equation}
where $p_2(y,\eta)=g^{ij}(\kappa^{-1}(y))\eta_i\eta_j$ and $d_{j\ell}\in S^{2j-\ell}$ are polynomial in the second variable. For $j=0$ there holds
\begin{equation}
    \label{eq.ecriture-symbole0}
    q_{-2}(\lambda;y,\eta)= \frac{1}{p_2(y,\eta) + 1 - \lambda},
\end{equation}
which is holomorphic in $\lambda \in \Gamma$.
\end{lemm}

\begin{proof} This is essentially \cite[Theorem 6.8]{Bouclet2012} applied with $h = |\lambda - 1|^{-\frac{1}{2}}$ and $z=z(\lambda):=\frac{\lambda -1}{|\lambda - 1|}$. In particular $|z(\lambda)| \leq 1$ and $\operatorname{dist}(z(\lambda),\mathbb{R}_+) \gtrsim 1$, which follows from the definition of $\Gamma$. With this choice of parameters, 
\[
    (L-\lambda) = |\lambda - 1|(-h^2\Delta_g - z), 
\]
so that one has to rescale the $q_{-2-j}$ constructed in \cite[Theorem 6.3]{Bouclet2012} by a factor $|\lambda - 1|^{-1}$. Then one has to use that $\Op_h(a(y,\eta))=\Op(a(y,h\eta))$, and incorporate the factors $h^j$ appearing in \cite[Theorem 6.3]{Bouclet2012}. The claimed symbolic properties can be then verified from \eqref{eq.ecriture-symbolej} and the fact
\[
    \sup_{\lambda \in \Gamma} \frac{\langle z(\lambda) \rangle}{\operatorname{dist}(z(\lambda),\mathbb{R}_+)} \leq C(\omega_0).
\]
Finally, that $q_{-2}$ is holomorphic in $\lambda$ is verified directly on the formula.
\end{proof}

\begin{proof}[Proof of \cref{lemma-pseudodiff-power}]
\noindent (\textit{i}) 
By functional calculus, for all $\lambda \in \Gamma$ there holds 
\begin{equation}
    \label{eq-resolvent-bound}
        \|(L-\lambda)^{-1}\|_{L^2\to L^2} = \sup_{z \in \sigma(L)}\left\vert\frac{1}{\lambda - z}\right\vert = \mathcal{O}(|\lambda|^{-1}).
\end{equation}
There also holds 
\[
    |(p_2(y,\eta) + 1) - \lambda| \gtrsim |\lambda| + |\eta|^2. 
\]
Write $\theta_k^{\kappa}=\theta_k\circ \kappa^{-1}$, and consider $\rho_0, \rho_1, \rho_2 \in C^{\infty}_c(V)$ such that $\rho_0 = 1$ on a neighborhood of $\supp \theta_2^{\kappa}$, and $\rho_{k+1} = 1$ on a neighborhood of $\supp \rho_k$, for $k\in\{0,1\}$. 

Consider the symbols $q_{-2-j}$ and $r_N$ constructed by \cref{lemma-parametrix} and write $Q_N(\lambda)=\Op(\sum_{j=0}^{N-1}q_{-2-j}(\lambda))$. Multiply \eqref{eq-bouclet} on the right by $\theta_2^{\kappa}$ and use that $\rho_0\theta_2^\kappa = \theta_2^\kappa$, apply $\kappa^*$ on the left and $\kappa_*$ on the right. This gives
\[
    (L-\lambda)\kappa^* \rho_1 Q_N(\lambda)\theta_2^{\kappa} \kappa_* = \theta_2 - \kappa^*\rho_2\Op(r_N(\lambda))\theta_2^{\kappa}\kappa_*.
\]
Then we multiply on the left by $\theta_1 (L-\lambda)^{-1}$ and apply $\kappa_*$ on the left and $\kappa^*$ on the right to get 
\[
    \kappa_*\theta_1(L-\lambda)^{-1}\theta_2 \kappa^* = \theta_1^{\kappa}Q_N(\lambda)\theta_2^{\kappa} + R_{N}(\lambda), 
\]
where
\[
    R_{N}(\lambda)=\kappa_*\theta_1(L-\lambda)^{-1}\kappa^*\rho_2\Op(r_N(\lambda))\theta_2^{\kappa}.
\]
It follows that 
\[
    \kappa_*\theta_1 L^{-\gamma}\theta_2 \kappa^* = \sum_{j=0}^{N-1}\theta_1^{\kappa}\Op(m_{-2\gamma -j})\theta_2^{\kappa} + \frac{i}{2\pi}\int_{\Gamma}\lambda^{-\gamma}R_{N}(\lambda)\mathrm{d}\lambda, 
\]
with 
\[
    m_{-2\gamma -j}(y,\eta) = \frac{i}{2\pi}\int_{\Gamma}\lambda^{-\gamma} q_{-2-j}(\lambda ; y,\eta)\mathrm{d}\lambda \in S^{-2\gamma-j}.
\]
Because $q_{-2-j} \in S_{\lambda}^{-2-j}$ the above integrals converge and define symbols belonging to $S^{-2\gamma-j}$. 
Note that for $j=0$, thanks to Cauchy's theorem we have 
\[
    m_{-2\gamma}(y,\eta)=\frac{i}{2\pi}\int_{\Gamma}\frac{\lambda^{-\gamma}}{p_2(y,\eta) +1 - \lambda}\mathrm{d}\lambda = (p_2(y,\eta) +1)^{-\gamma}. 
\]
Note also that $r_N \in S^{-N}_{\lambda}$ provides enough decay to show convergence of the remainder integral as an operator $L^2 \to L^2$.

Next, from Borel's summation theorem \cite[Theorem 4.11]{EvansZworski2003} there exists $m \in S_{\rm cl}^{-2\gamma}$ satisfying the asymptotic expansion $m \sim \displaystyle\sum_{j=0}^{\infty} m_{-2\gamma - j}$. Compute for any $N\geq 1$, 
\[
    R :=\kappa_*\theta_1 L^{-\gamma}\theta_2 \kappa^* - \theta_1^{\kappa}\Op(m)\theta_2^{\kappa} = - \theta_1^{\kappa}\Op\left(m - \sum_{j=0}^{N-1} m_{-2\gamma -j}\right)\theta_2^{\kappa} + \frac{i}{2\pi}\int_{\Gamma}\lambda^{-\gamma}R_{N}(\lambda)\mathrm{d}\lambda.
\]
We claim that $R \in \Op(S^{-\infty})$, or equivalently that its Schwartz kernel $K$ is smooth. This is the only thing which remains to be explained at this stage. In order to do that, we are going to prove that all derivatives $\partial^{a}_y\partial_z^bK(y,z)$ are locally bounded. Fix such $a$ and $b$, and fix some compact neighborhood $W_1 \times W_2$ of a given $(y_0,z_0) \in \Rr^d \times \Rr^d$.

First, since $m_{\geq N} := m - \displaystyle\sum_{j=0}^{N-1} m_{-2\gamma -j} \in S^{-2\gamma - N}$, denoting $K_{\geq N}$ the kernel of $\Op(m_{\geq N})$ and because
\[
    \partial_y^{a}\partial_z^b K_{\geq N}(y,z)=\int_{\mathbb{R}^d} \partial_y^a\partial_z^b \left(e^{i(y-z)\cdot \eta}m_{\geq N}(y,\eta)\right)\frac{\mathrm{d}\eta}{(2\pi)^d}
\] 
we can use the crude estimate:
\[
    |\partial_y^{a}\partial_z^b K_{\geq N}(y,z)| \leq C \int_{\mathbb{R}^d} \langle \eta \rangle^{-2\gamma - N +|a| +|b|} \mathrm{d}\eta \lesssim 1
\]
which holds as soon as $-2\gamma - N +|a| +|b| < -d $, so for instance assume $N > d + |a| + |b| - 2\gamma$.

It remains to show boundedness on $W_1 \times W_2$ for $\partial_y^a\partial_z^bK_{\rm rem}(y,z)$, where $K_{\rm rem}$ is the kernel of $\int_{\Gamma}\lambda^{-\gamma}R_{N}(\lambda)\mathrm{d}\lambda$. Write $K_N(\lambda, \cdot, \cdot)$ for the kernel of $R_{N}(\lambda)$ so that 
\[
    \partial_y^a\partial_z^b K_{\rm rem}(y,z) = \int_{\Gamma} \lambda^{-\gamma} \partial_y^a\partial_z^bK_N(\lambda,y,z) \mathrm{d}\lambda.
\]
We compute: 
\[
    K_{N}(\lambda,y,z)=\theta_1^{\kappa}(y)((L-\lambda)^{-1}(\kappa^*\mathfrak{K}_N(\lambda;\cdot,z)))(\kappa^{-1}(y)), 
\]
where $\mathfrak{K}_N(\lambda; \cdot, \cdot)$ is the Schwartz kernel of $\rho_2\Op(r_N(\lambda))\theta_2^{\kappa}$. Our goal is then reduced to proving
\begin{equation}
    \label{eq-smoothness-sufficient}
    \int_\Gamma |\lambda|^{-\gamma} \sup_{(y,z)\in W_1\times W_2} \left|\partial_y^a\partial_z^b K_N(\lambda;y,z)\right| \mathrm{d}\lambda <\infty .
\end{equation}
First estimate the kernel $\mathfrak K_N$. Since $r_N(\lambda)\in S_{\lambda}^{-N}$, we can assume that for some fixed $A, B \gg 1$ there holds: 
\[
    |\partial_y^a\partial_\eta^b r_N(\lambda;y,\eta)| \leq C(a,b,A,B)\langle\eta\rangle^{-A}\langle\lambda\rangle^{-B}.
\]
This is achieved by taking $N > A+2B$ for instance. 
Fix an even integer $s = 2m > |a| + \frac{d}{2}$. We claim that 
\begin{equation}
    \label{eq-smoothness-claim1}
    \sup_{z\in W_2} \|\partial_z^b\mathfrak{K}_N(\lambda;\cdot,z)\|_{H^s} \leq C(s,b,B)\langle\lambda\rangle^{-B}.
\end{equation}
Since 
\[ 
    \mathfrak{K}_N(\lambda;y,z) = (2\pi)^{-d}\rho_2(y)\theta_2^\kappa(z) \int_{\Rr^d}e^{i(y-z)\cdot\eta} r_N(\lambda;y,\eta)\mathrm{d}\eta,
\]
differentiating $\partial_y^\sigma\partial_z^b\mathfrak{K}_N$ with $|\sigma|\leq s$ gives a finite sum of terms bounded by
\[
    \int_{\Rr^d} \langle\eta\rangle^{|\sigma|+|b|}|\partial_y^{\sigma'}r_N(\lambda;y,\eta)|\mathrm{d}\eta \lesssim \int_{\Rr^d} \langle\eta\rangle^{|\sigma|+|b|-A}\mathrm{d}\eta , \qquad |\sigma'| \leq |\sigma|, 
\]
therefore choosing $A>s+|b|+d$ is enough to ensure convergence of the integral, so that 
\[
    |\partial_y^a\partial_z^b\mathfrak{K}_N(\lambda;y,z)| \leq C(s,b,B)\langle \lambda \rangle^{-B}.
\]
Since $\mathfrak{K}_N(\lambda;y,z)$ has compact support, we finally obtain \eqref{eq-smoothness-claim1}. 
Next, we set 
\[
    f_{\lambda,z,b}(y) := \partial_z^b\mathfrak{K}_N(\lambda;\cdot,z), \qquad  v_{\lambda,z,b} := \kappa_*\left((L-\lambda)^{-1}\kappa^*f_{\lambda,z,b}\right),
\]
which is such that $\partial_z^bK_N(\lambda;y,z) = \theta_1^\kappa(y)v_{\lambda,z,b}(y)$ and satisfies locally in $V$,
\[
    \kappa_*L\kappa^*v_{\lambda,z,b} =f_{\lambda,z,b} + \lambda v_{\lambda,z,b}.    
\]
The operator $\kappa_*L\kappa^*$ is uniformly elliptic on compact subsets of $V$, therefore by iterating $m$ times elliptic regularity in Sobolev spaces, we obtain 
\[
    \|\theta_1^{\kappa} v_{\lambda,z,b}\|_{H^{s}} \leq C(s) \langle \lambda \rangle^m(\|\chi_2 v_{\lambda,z,b}\|_{L^2} + \|f_{\lambda,z,b}\|_{H^s}),
\]
where $\chi_2 \in C^{\infty}_c(V)$ is equal to $1$ on $\supp \theta_1^\kappa$.
Using \eqref{eq-resolvent-bound} we have 
\[
    \|\chi_2 v_{\lambda,z,b}\|_{L^2(\Rr^d)} \leq C \|(L-\lambda)^{-1}\kappa^*f_{\lambda,z,b}\|_{L^2(M)} \leq \langle \lambda \rangle^{-1} \|f_{\lambda,z,b}\|_{L^2(\Rr^d)}.
\]
It follows that 
\[
    \|\theta_1^{\kappa} v_{\lambda,z,b}\|_{H^{s}} \leq C(s)\langle \lambda \rangle^m\|f_{\lambda,z,b}\|_{H^s}.
\]
Using Sobolev's embedding and \eqref{eq-smoothness-claim1} yields 
\[
    \sup_{y \in W_1} |\partial_y^a\theta_1^\kappa v_{\lambda,z,b}|\leq C\|\theta_1^\kappa v_{\lambda,z,b}\|_{H^s} \leq C(s,a,b,B)\langle\lambda\rangle^{m-B}, 
\]
which implies 
\[
    \sup_{(y,z)\in W_1\times W_2} | \partial_y^a\partial_z^b K_N(\lambda;y,z)|\leq
    C(a,b,B)\langle\lambda\rangle^{m-B}.
\]
It now suffices to choose $B > m+\gamma +1$ to infer \eqref{eq-smoothness-sufficient}.

We have thus obtained that $\kappa_*\theta_1L^{-\gamma}\theta_2\kappa^*$ is a classical pseudo-differential operator of order $-2\gamma$. Choose $\vartheta\in C_c^\infty(U)$ with $\vartheta=1$ on a neighborhood of $\supp\theta_1$. Since $L^k$ is local we write $\theta_1L^k=\theta_1L^k\vartheta$. Composing with this differential of order $2k$ with the pseudo-differential operator of order $-2\gamma$, gives a pseudo-differential operator of $2k-2\gamma=2\beta $ and principal symbol $p_2^k p_2^{-\gamma}=p_2^\beta$ as claimed.

\medskip
\noindent (\textit{iii})
With the same decomposition of~\(L^\beta\) as in (\textit{i}), it is enough to treat negative $\beta$, so we write $\beta=-\gamma$ with $\gamma >0$. The Schwartz kernel $K$ of the operator $\rho_1L^{-\gamma}\rho_2$ has compact support, our task therefore reduces to proving smoothness of the Schwartz kernel, itself equivalent to proving that for all $m, n \geq 0$, $\rho_1L^{-\gamma}\rho_2 : H^{-n}(M) \to H^m(M)$, which is implied by 
\begin{equation}
\label{eq.target-smoothing}
    \|\rho_1(L-\lambda)^{-1}\rho_2\|_{H^{-n}(M)\to H^m(M)}\leq C(n,m) \langle\lambda\rangle^{-1}, \qquad \lambda\in \Gamma.
\end{equation}
Fix $k\geq m+n$. Since $\rho_2=0$ on a neighborhood of
$\supp\rho_1$, we may choose smooth functions
$\zeta_0,\dots,\zeta_k$ such that~\(\zeta_0 = \rho_1\), that for all~\(j=1,\hdots, k\), we have $\zeta_j=1$ on a neighborhood of  $\supp \zeta_{j-1}$ and $\zeta_j=0$ on a neighborhood of~\(\supp \rho_2\). Set~\(A_j= \zeta_j (L-\lambda)^{-1} \rho_2\). 
Then
\[
(L-\lambda) A_j = [L,\zeta_j] (L-\lambda)^{-1} \rho_2 + \zeta_j \rho_2.
\]
We have that~\(\zeta_j \rho_2 = 0\) and since~\(L\) is local,~\([L,\zeta_j] = [L,\zeta_j] \zeta_{j+1}\). We deduce
\[
A_j = (L-\lambda)^{-1} [L,\zeta_j] A_{j+1}
\]
and by induction
\begin{equation}\label{eq.iterative-decompo}
\rho_1 (L-\lambda )^{-1} \rho_2 = A_0 = \prod_{j=0}^{k-1} (L-\lambda)^{-1}[L,\zeta_j] A_k.
\end{equation}

For \(\lambda\in\Gamma\), we have the spectral estimate 
\[
    \|(L-\lambda)^{-1}\|_{H^s\to H^s}\leq  C(s) \langle\lambda\rangle^{-1}
\]
and the elliptic regularity estimate
\[
    \|(L-\lambda)^{-1}\|_{H^s\to H^{s+2}}\leq C(s)\left(1+|\lambda| \|(L-\lambda)^{-1}\|_{H^s\to H^{s}}\right) \leq C(s).
\]
The commutators are estimated using 
\[
    \|[L,\zeta_j]\|_{H^s \to H^{s-1}} \leq C(s).
\]
Using these estimates in \eqref{eq.iterative-decompo} gives 
\[
    \|\rho_1(L-\lambda)^{-1}\rho_2\|_{H^{-n}\to H^{-n+k}} \leq C(n,k) \langle \lambda \rangle^{-1}.
\]
It therefore remains to choose $k>n+m$ to obtain \eqref{eq.target-smoothing}.

\medskip
\noindent (\textit{ii})
Let $A=\rho_1 L^{\beta} \rho_2$. 

Choose finitely many relatively compact coordinate charts~\(\{U_j, \kappa_j\}_{j=1, \dots, J}\) covering $\operatorname{supp}\rho_1$, and choose functions $\chi_j \in \mathcal{C}^{\infty}_c(U_j)$ such that 
\[
    \sum_{j = 1}^J \chi_j = \rho_1 .
\]
Write $\kappa_j : U_j \rightarrow V_j \subseteq \Rr^d$ for the coordinate map.

Choose also $\eta_j \in \mathcal{C}_c^{\infty}(U_j)$ such that $\eta_j = 1$ on a neighborhood of $\supp \chi_j$. Then we can write 
\[
    A = \sum_{j = 1}^J \chi_j L^{\beta} \eta_j \rho_2 + \sum_{j = 1}^J  \chi_j L^{\beta} (1-\eta_j) \rho_2.
\]

For the term $ \chi_j L^{\beta} (1-\eta_j) \rho_2$, observe that 
\[
    \supp (\chi_j) \cap \supp ( (1-\eta_j) \rho_2) = \varnothing, 
\]
so that~(\textit{iii}) shows that this operator is trace-class.

For the term $ \chi_j L^{\beta} \eta_j \rho_2$, observe that both $\chi_j$ and $\eta_j\rho_2$ are compactly supported in the same chart $U_j$. We can therefore consider the operator
\[
    A_j = (\kappa_j)_* \chi_j L^{\beta} \eta_j\rho_2\kappa_j^*.
\]
In view of (\textit{i}), it is a pseudo-differential operator of order $2\beta$.
Let $\zeta_1, \zeta_2 \in \mathcal{C}^{\infty}_c(V_j)$ denote functions such that $\zeta_1 = 1$ on a neighborhood of $\supp (\chi_j) \circ \kappa_j^{-1}$ and $\zeta_2 = 1$ on a neighborhood of $\supp(\eta_j \rho_2) \circ \kappa_j^{-1}$. Fix a real number $s$ satisfying $\frac{d}{2}<s<-\beta$ and write
\[
    A_j = (\zeta_1\langle D \rangle^{-s}) \langle D \rangle^{s}A_j\langle D \rangle^{s}(\langle D \rangle^{-s}\zeta_2)
\]
where~\(\langle D \rangle =(\operatorname{id}-\Delta)^{\frac{1}{2}}\). 
The operator $\langle D \rangle^{s}A_j\langle D \rangle^{s}$ is a pseudo-differential operator of order $2s+2\beta < 0$, therefore continuous $L^2(\Rr^d) \to L^2(\Rr^d)$. Since $2s>d$, the operator $\zeta_1\langle D \rangle^{-s}$ is Hilbert--Schmidt. It follows that $A_j \in \mathfrak{S}_1(L^2(\Rr^d))$.  

In the coordinate chart $\kappa_j$ the volume form $\dVolg$ writes $\sqrt{\det g_{ij}(Y)}\mathrm{d}Y$, and since we are working on compact subsets of $U_j$, $\sqrt{\det g_{ij}(Y)}$ is bounded above and below by positive constants. It follows that $\kappa_j^* : L^2(V_j,\mathrm{d}Y) \to L^2(U_j,\dVolg)$ and $(\kappa_j)_* : L^2(U_j,\dVolg) \to L^2(V_j,\mathrm{d}Y)$ are bounded. It follows that $\chi_j L^{\beta}\eta_j\rho_2 = \kappa_j^* A_j (\kappa_j)_*$ is trace-class.

\medskip
\noindent (\textit{iv})
If $M$ is compact, we only use~(\textit{iii}) seeing that~\(1-\rho_2\) is compactly supported away from the support of~\(\rho_1\).

If $\beta$ is a natural integer, then the operator $L^{\beta}$ is a differential operator, thus local, and so for all functions $f$, $\supp (L^{\beta}(1-\rho_2)f) \subset \supp ((1-\rho_2)f)$, which does not intersect $\supp \rho_1$. We deduce $\rho_1 L^{\beta}(1-\rho_2)=0$. 

In the Euclidean case, since $L^{\beta}$ is a classical pseudo-differential operator on $\Rr^d$, the operator $\rho_1 L^{\beta}(1-\rho_2)$ has kernel given by 
\[
    K(x,y)= \rho_1(x)(1-\rho_2(y))\int_{\Rr^d}e^{i(x-y)\cdot \xi} (1+|\xi|^2)^\beta \frac{\mathrm{d}\xi}{(2\pi)^d}, 
\]
which already shows that $K$ is supported in a region for which $|x-y| \geq c >0$ and compactly supported in the variable $x$. Let $\mathcal{K}$ be such a compact. This is due to the disjointness of the supports of $\rho_1$ and~\(1- \rho_2\). It is thus enough to prove that 
\[
    |K(x,y)|\leq C|x-y|^{-d-1}, \qquad x\in \mathcal{K}, \quad |y-x| \geq c, 
\]
which then ensures that $K \in L^2(\mathbb{R}^d \times \mathbb{R}^d)$ and therefore $\rho_1 L^{\beta}(1-\rho_2) \in \mathfrak{S}_2(L^2(M))$.

We perform multiple integration by parts. More precisely, fix $x \in \supp \rho_1$ and $y\in \supp (1-\rho_2)$ so that $|x-y|\geq c >0$ and let $\mathcal{L}=-i|x-y|^{-2} (x-y) \cdot \nabla_\xi$, which satisfies 
\[
    \mathcal{L}e^{i(x-y)\cdot \xi}=e^{i(x-y)\cdot \xi}, \qquad \mathcal{L}^*= -\mathcal{L}.
\]
Using the triangle inequality we bound 
\begin{multline*}
    |K(x,y)|=\rho_1(x)(1-\rho_2(y))\left|\int_{\Rr^d}e^{i(x-y)\cdot \xi} (\mathcal{L}^*)^M\left((1+|\xi|^2)^\beta\right) \frac{\mathrm{d}\xi}{(2\pi)^d}\right| \\
    \leq C(M)|x-y|^{-M} \int_{\Rr^d}\langle  \xi \rangle^{2\beta - M} \mathrm{d}\xi \leq C(M)|x-y|^{-d-1}
\end{multline*}
as soon as $M > \max\{2\beta + d,d+1\}$.
\end{proof}

\section{Proof of the main results}
\label{sec.proof}

\subsection{\texorpdfstring{Proof of \cref{thm.riemannian}}{Proof of main theorem}}

The proof of (\textit{iii}) $\Longrightarrow$ (\textit{ii}) of \cref{thm.riemannian} follows by observing that isometries commute with the Laplace--Beltrami operator, see \cite{Canzani2013}.

The proof of implication (\textit{i}) $\Longrightarrow$ (\textit{iii}) of~\cref{thm.riemannian} relies on pseudo-differential calculus in~\(\Rr^d\). For this, we use charts. We recall that \emph{in fine}, we prove a local characterisation of the vector field~\(u\), namely that for all~\(x_0 \in M\), almost every~\(t_0 \in \Rr\) and all~\((\mu,\nu) \in \{1, \dots , d\}^2\) we have
\[
    \nabla_\mu u_\nu(t_0,x_0) + \nabla_\nu u_\mu(t_0,x_0) = 0.
\]
In view of \cref{lemma:killing}, it is enough to prove that for all~\(x_0 \in M\) and all~\(t_0 \in \Rr\), the map 
\[
    \mathrm{D}X_{t_0+s,t_0}(x_0) : \left(T_{x_0}M,g(x_0)\right) \longrightarrow \left(T_{X_{t_0+s,t_0}(x_0)}M,g(X_{t_0+s,t_0}(x_0))\right)
\]
is an isometry for all $|s|<\delta$, for some $\delta = \delta(t_0,x_0)>0$. 

We fix~\(x_0 \in M\) and~\(t_0 \in \Rr\) for the rest of this section. Consider a local chart
\[
    \kappa : U \subseteq M \rightarrow V
\]
such that~\(U\) is a neighborhood of~\(x_0\). Without loss of generality, we assume that~\(\kappa(x_0) = 0\) and thus~\(V\) is a neighborhood of~\(0\). 

Around~\(x_0\), we only use the chart~\(\kappa\). Therefore, one needs to localize the operators we work with in~\(U\) and more precisely in subsets of~\(U\) that allow to apply~\cref{lem:compo-symbole}. We start with setting up some notation.

There exists~\(r>0\) such that the ball~\(B(0,r)\) of center~\(0\) and radius~\(r\) in~\(\Rr^d\) is included in~\(V\). Set
\[
    V_0 = B\left(0,\frac{3r}{4}\right), \qquad U_0 = \kappa^{-1}(V_0).
\]
Since~\(U_0\) is bounded and~\(\overline{U_0} \subset U\), we deduce that
\[
    s \mapsto \sup_{x\in \overline{U_0}}\{d_g(X_{t_0 + s,t_0}(x), x) + d_g(X_{t_0 + s,t_0}^{-1}(x), x)\}
\]
is continuous (where~\(d_g\) is the geodesic distance). Therefore, there exist~\(\tau > 0\) and $\varepsilon_0>0$ such that for all~\(s \in [-\tau,\tau]\) we have
\[
B_g(x_0,\varepsilon_0) \subset X_{t_0 + s,t_0}(\overline{U_0}) \subseteq U
\]
where~\(B_g(x_0,\varepsilon_0)\) is the geodesic ball of center~\(x_0\) and radius~\(\varepsilon_0\) (remark that by compactness of~\(\overline{U_0}\), there exists~\(\varepsilon >0\) such that for all~\(y \in \overline{U_0}\), we have~\(B_g(y,\varepsilon) \subseteq U\)). We can define
\[
\chi_{t_0 + s,t_0} = \kappa \circ X_{t_0 + s,t_0} \circ \kappa^{-1} : V_0 \rightarrow V.
\]
Remark that~\(\chi_{t_0+s,t_0}\) is differentiable with differential
\[
    \Jac\chi_{t_0+s,t_0} = \Jac\kappa_{|X_{t_0 + s,t_0} \circ \kappa^{-1}} (\Jac X_{t_0 + s,t_0} )_{\kappa^{-1}} \Jac\kappa^{-1}.
\]
Since~\(s\mapsto \displaystyle\sup_{\overline{U_0}}\|\Jac X_{t_0 + s,t_0} - I_d\|\) is continuous and since~\(D\kappa\) and~\(D\kappa^{-1}\) are uniformly continuous, we get that there exists~\(\tau_0 \in (0,\tau]\) such that
\[
    \sup_{\substack{x\in V_0 \\ s\in [-\tau_0,\tau_0]}} \|\Jac\chi_{t_0 + s,t_0}(x) - I\| \leq \frac14.
\]
Set~\(V_1 = \displaystyle\bigcap_{|s|\leq \tau_0}\operatorname{range} \chi_{t_0+s,t_0}\), which is a neighborhood of~\(0\) as it contains a ball $B(0,\varepsilon'_0)$. We have that~\(\chi_{t_0+s,t_0} : V_0 \rightarrow V\) defines a~\(\mathcal C^1\) diffeomorphism on its range~(that includes~\(V_1\)) such that for all~\(s\in [-\tau_0,\tau_0]\) and all~\(x\in V_1\)
\[
    \|\Jac\chi_{t_0 + s,t_0}^{-1} (x) -I_d\| = \|(\Jac\chi_{t_0 + s,t_0} (\chi_{t_0 + s,t_0}^{-1} (x)))^{-1} - I_d\| \leq \frac13.
\]

For the rest of this section, let~\(V_2\) be a convex open subset of~\(V_0 \cap V_1\) that contains~\(0\). We fix~\(s\in [-\tau_0,\tau_0]\).  Let~\(\theta_0,\theta_1,\theta_3\) be smooth functions compactly supported in~\(\kappa^{-1}(V_2)\) such that: \(\theta_0 = 1\) in a neighborhood of~\( X_{t_0+s,t_0}(x_0)\); \(\theta_1 = 1\) on a neighborhood of the support of~\(\theta_0\);~\(\theta_3 = 1\) on a neighborhood of the support of~\(\theta_1\) and set~\(\theta_2 = \theta_3\circ X_{t_0+s,t_0} \).

Note that the time~\(\tau_0\), the open set~\(V_2\) and functions~\(\theta_i\) depend implicitly on~\(x_0\) and~\(t_0\).

Finally, we introduce the following operators, for all~\(i,j \in \{0,1,2,3\}\), acting on a test function~\(f\) of~\(V_2\): 
\begin{align*}
    P^{-\mez}_{i,j} f &= \kappa_* \theta_i P^{-\mez}\theta_j \kappa^* = (\theta_i P^{-1/2} (\theta_j f\circ \kappa))\circ \kappa^{-1}, \\
    P_{i,j} f &= \kappa_* \theta_i P\theta_j \kappa^* =  (\theta_i P (\theta_j f\circ \kappa))\circ \kappa^{-1}.
\end{align*}
We recall that we understand~\(\theta_j f\circ \kappa\) as the (smooth) extension of~\(\theta_j f\circ \kappa\) equal to~\(0\) outside the support of~\(\theta_j\).

Since both $t_0$ and $s$ are now fixed, write $S$, $X$ and $\chi$ in place of $S(t_0+s,t_0)$, $X_{t_0+s,t_0}$ and $\chi_{t_0+s,t_0}$.

\begin{lemm} With the above notations, the operator
\[
    \theta_0 P^{-\mez} \theta_1 \left(S\theta_2 P \theta_2 S^{-1} -  \theta_3 P\theta_3 \right) \theta_1 P^{-1/2} \theta_0
\]
is Hilbert--Schmidt on~\(L^2(M)\). This translates into
\[
    P_{0,1}^{-\mez} \left(S_\kappa P_{2,2} S_\kappa^{-1}  - P_{3,3}\right)P^{-\mez}_{1,0}
\]
is Hilbert--Schmidt on~\(L^2(V_2)\), where~\(S_\kappa\) and~\(S_\kappa^{-1}\) are respectively the compositions with~\(\chi^{-1}\) and~\(\chi\).
\end{lemm}

\begin{proof} Our starting point is \cref{prop:criterion_for_abscont}, from which we know that 
\[
    \theta_0 P^{-\mez}\left(SPS^{-1} - P \right) P^{-1/2} \theta_0 \in \mathfrak{S}_2(L^2(M)).
\]

Write $\theta_0P^{-\mez} = \theta_0P^{-\mez}\theta_1 + \theta_0P^{-\mez}(1-\theta_1)$. By \cref{as} and because $\theta_1 = 1$ on a neighborhood of $\supp \theta_0$, it follows that $\theta_0 P^{-\mez}(1-\theta_1)$ is a Hilbert--Schmidt operator. 
Moreover, the operator $(SPS^{-1}-P)P^{-\mez}\theta_0$ is continuous as an operator $L^2(M) \to L^2(M)$. It follows that 
\[
    \theta_0 P^{-\mez}(1-\theta_1)(SPS^{-1}-P)P^{-\mez}\theta_0 \in \mathfrak{S}_2(L^2(M)), 
\]
and therefore that 
\[
    \theta_0 P^{-\mez}\theta_1(SPS^{-1}-P)P^{-\mez}\theta_0 \in \mathfrak{S}_2(L^2(M)).
\]
Repeating this analysis on $P^{-\mez}\theta_0$ yields that 
\[
    \theta_0 P^{-\mez}\theta_1(SPS^{-1}-\theta_3 P\theta_3)\theta_1P^{-\mez}\theta_0 \in \mathfrak{S}_2(L^2(M)).
\]
In the above we have also used that $\theta_3\theta_1=\theta_1$. Observe that for all test functions $f$, since~\(\theta_1 \theta_2 \circ X^{-1} = \theta_1\), we have
\[
    \theta_1 S f  = \theta_1 f \circ X^{-1} = \theta_1 \theta_2 \circ X^{-1} f \circ X^{-1} = \theta_1 S \theta_2f,
\]
which means $\theta_1 S = \theta_1 S \theta_2$ and allows one to conclude that 
\[
    \theta_0 P^{-\mez} \theta_1 \left(S\theta_2 P \theta_2 S^{-1} -  \theta_3 P\theta_3 \right) \theta_1 P^{-1/2} \theta_0 \in \mathfrak{S}_2(L^2(M)).
\]
Since composition by $\kappa$ or $\kappa^{-1}$ is a continuous operation between $L^2(\kappa^{-1}(V_2))$ and $L^2(V_2)$ (recall that~\(\theta_0\) is supported in~\(\kappa^{-1}(V_2)\)), it follows that 
\[
   \kappa_*\theta_0 P^{-\mez} \theta_1 \left(S\theta_2 P \theta_2 S^{-1} -  \theta_3 P\theta_3 \right) \theta_1 P^{-1/2} \theta_0\kappa^* \in \mathfrak{S}_2(L^2(V_2)). 
\]
It remains to observe that
\[
    P_{0,1}^{-\mez} \left(S_\kappa P_{2,2} S_\kappa^{-1}  - P_{3,3}\right)P^{-\mez}_{1,0} = \kappa_*\theta_0 P^{-\mez} \theta_1 \left(S\theta_2 P \theta_2 S^{-1} -  \theta_3 P\theta_3 \right) \theta_1 P^{-1/2} \theta_0\kappa^*. \qedhere 
\]
\end{proof}

Our next step is to identify the leading order operator of $P_{0,1}^{-\mez} \left(S_\kappa P_{2,2} S_\kappa^{-1}  - P_{3,3}\right)P^{-\mez}_{1,0}$ as a pseudo-differential operator.

\begin{lemm}\label{lemma:main-symbol}
Let \(q_0  \in S^0\) be defined for all~\(y\in V_2\) and~\(\eta \in \Rr^d\) by 
\begin{equation}
    \label{eq.principal-symbol}
    q_0(y,\eta) = \omega(\eta)\theta_0^2(\kappa^{-1}(y))\left( \frac{|\eta|^{2\alpha}_{y}}{|\Jac \chi(\chi^{-1}(y))^\top\eta|_{\chi^{-1}(y)}^{2\alpha}}-1\right), 
\end{equation}
where $|\eta|^2_{z}:=g^{ij}(\kappa^{-1}(z))\eta_i\eta_j$, and with $\omega$ which is a radial function such that $\omega = 0$ near $\eta =0$ and $\omega (\eta) = 1$ for $|\eta| \geq C$ for some constant $C>0$. 
There exists~\(\delta > 0\) and~\(r \in S^{-\delta}\) such that
\[
    P_{0,1}^{-\mez} (S_\kappa P_{2,2} S_\kappa^{-1}  - P_{3,3})P^{-\mez}_{1,0} = \Op(q_0) + \Op(r).
\]
\end{lemm}

\begin{proof} In view of \cref{lemma-pseudodiff-power} (\textit{i}) we can write $P_{2,2} = \Op(a_2)$ for some $a_2 \in S^{-2\alpha}_{\rm cl}$ whose principal symbol is $p_{2,2}(y,\eta)=\theta_2^2\circ\kappa^{-1}(y)|\eta|_y^{-2\alpha}$. 

The restriction of~\(\chi^{-1}\) to~\(V_2\) defines a~\(\mathcal C^1\) diffeomorphism on its range~\(W\). We consider the restriction of~\(\chi\) to~\(W\). Take $\tilde{\theta}_1$ supported in $V_2$ such that $\tilde{\theta}_1 = 1$ on $\supp \theta_1 \circ \kappa^{-1}$. Remark that~\(\theta_2 \circ \kappa^{-1}\) is supported in~\(\chi^{-1}(V_2) = W\) and take~\(\tilde{\theta}_2\) supported in~\(W\) such that~\(\tilde \theta_2 = 1\) on a neighborhood of~\(\supp \theta_2\circ \kappa^{-1} \cup \chi^{-1}(\supp \tilde \theta_1)\). Apply \cref{lem:compo-symbole} with $(\psi,\rho)=(\tilde{\theta}_1,\tilde \theta_2)$ yields 
\[
    \tilde{\theta}_1S_\kappa P_{2,2} S_\kappa^{-1}\tilde{\theta}_1 = \tilde{\theta}_1 C_{\chi}^{-1}\tilde{\theta}_2 P_{2,2} \tilde{\theta}_2 C_{\chi}\tilde{\theta}_1 = \tilde{\theta}_1\Op(b)\tilde{\theta}_1 + \Op(r_1), 
\]
where $b(y,\eta)=\omega_1(\Jac \chi (\chi^{-1}(y))^\top\eta)p_{2,2}\left(\chi^{-1}(y),\Jac \chi (\chi^{-1}(y))^\top\eta\right)$ for all $(y,\eta) \in V_2\times\Rr^d$ and for some radial $\omega_1$ which equals one for $|\eta| \geq C$, $C>0$ large enough; and with $r_1 \in S^{-2\alpha - \delta}$ for some $\delta >0$. Note that in the case of smooth $\chi$ we can take $\delta =1$. Applying \cref{lemma-pseudodiff-power} (\textit{i}) again yields
\[
     \tilde{\theta}_1\left(S_\kappa P_{2,2} S_\kappa^{-1} - P_{3,3}\right)\tilde{\theta}_1 = \tilde{\theta}_1\Op(b - a_3)\tilde{\theta}_1 + \Op(r_2), 
\]
with $a_3 \in S^{-2\alpha}_{\rm cl}$ whose principal symbol is given by $p_{3,3}(y,\eta)=\omega_1(\eta)\theta_3^2\circ\kappa^{-1}(y)|\eta|_y^{-2\alpha}$; and where $r_2 \in S^{-2\alpha -\delta}$. Another application of \cref{lemma-pseudodiff-power} yields that $P^{-\mez}_{0,1}=\Op(c_L)$, $P^{-\mez}_{1,0}=\Op(c_R)$ with $c_L, c_R \in S^{\alpha}_{\rm cl}$ with principal symbol both given by $p^{-\mez}_{0,1}(y,\eta)=\omega_1(\eta)\theta_0(\kappa^{-1}(y))\theta_1(\kappa^{-1}(y))|\eta|_y^{\alpha}$. Write 
\begin{align*}
    P_{0,1}^{-\mez}(S_\kappa P_{2,2} S_\kappa^{-1}  - P_{3,3})P^{-\mez}_{1,0} &= P_{0,1}^{-\mez}\tilde{\theta}_1(S_\kappa P_{2,2} S_\kappa^{-1}  - P_{3,3})\tilde{\theta}_1P^{-\mez}_{1,0} \\ 
    &= \Op(c)\left( \Op(b - a_{3}) + \Op(r_2)\right)\Op(c).
\end{align*}
First, since $c \in S^{\alpha}$ and $r_2\in S^{-2\alpha -\delta}$ it follows from \cref{prop.compo.pseudo} that
\[
\Op(c)\Op(r_2)\Op(c) \in \Op(S^{-\delta}).
\]
Applying \cref{prop.compo.pseudo} again yields 
\[
    \Op(c)\Op(b - a_{3})\Op(c) = \Op\left(p^{-\mez}_{0,1}(b-p_{3,3})p^{-\mez}_{0,1}\right) + \Op(r_3) = \Op(q_0) + \Op(r_3)
\]
for some $r_3 \in S^{-\delta}$, and where $q_0=p^{-\mez}_{0,1}(b-p_{3,3})p^{-\mez}_{0,1} \in S^0$. It remains to observe that for $|\eta| \geq C$ large enough there holds $\omega_1(\eta)=\omega_1(\Jac \chi (\chi^{-1}(y))^\top\eta)=1$ and for $y \in V_2$, due to the definitions of the functions $\theta_j$, we obtain the claimed formula for $q_0(y,\eta)$.
\end{proof}

Our next step is to analyse the Hilbert--Schmidt condition at the level of the main operator that we have identified.

\begin{lemm}\label{lemma:contradiction}
Let~\(q_0 \in S^0\) such that for all~\(\eta \in \Rr^d\) satisfying~\(|\eta| \geq C_0\) and all~\(y \in V_2\), we have
\[
    q_0 (y,\eta) = q_0\left(y,\frac{\eta}{|\eta|}\right),
\]
and let \(r \in S^{-\delta}\) for some $\delta >0$. Assume that
\[
    \Op (q_0) + \Op(r) \in \mathfrak{S}_2(L^2(V_2)).
\]
Then~\(q_0(y,\eta) = 0\) for all $y \in V_2$ and all $|\eta| \geq 1$.
\end{lemm}

\begin{proof}
Consider a partition of unity $1 = \sum_{N \in 2^{\mathbb{N}}} \psi_N(\eta)$ where for $N>1$, $\psi_N$ is radial, supported in $B(0,4N)\setminus B(0,\frac{N}{2})$ and equals $1$ on $B(0,2N) \setminus B(0,N)$. Let $Q_N=\Op(q_0(y,\eta) \psi_N(\eta))$ and $R_N=\Op(r\psi_N(\eta))$ so that
\[
    T:= \Op (q_0) + \Op(r) = \sum_{N \in 2^{\mathbb{N}}} Q_N + R_N.
\]
Since the Schwartz kernel of a pseudo-differential operator $A=\Op(a)$ is given by $K_A(x,y)=\check{a}(x,x-y)$, the inverse Fourier transform being taken in the second variable only, it follows that we can compute 
\[
    \|R_N\|^2_{\mathfrak{S}_2(L^2(V_2))} \leq \iint_{V_2 \times \Rr^d}|\psi_N(\eta)|^2|r(y,\eta)|^2 \mathrm{d}y\mathrm{d}\eta \lesssim |V_2| \int_{\frac{N}{2} \leq |\eta| \leq 4N}  \frac{\mathrm{d}\eta}{\langle \eta\rangle^{2\delta}} \lesssim N^{d-2\delta}.
\]
where we have used that $r\in S^{-\delta}$. 

Assume that there exists $(y_0,\eta_0)$ with $|\eta _0| \geq C_0$ such that $q_0(y_0,\eta_0)\neq 0$. By continuity of $(y,\eta)\mapsto q_0(y,\eta)$ and homogeneity in $\eta$ there exists $\delta >0$ such that we can bound for any $N \gg 1$,  
\[
    |q_0(y,\eta)| \geq c >0, \qquad y \in \overline{B(y_0,\delta)}, \eta \in \overline{B}(N\eta_0,N\delta).
\]
Let $C_N = \overline{B}(N\eta_0,N\delta) \cap \left(B(0,2N)\setminus B(0,N)\right)$ and observe that $|C_N| \gtrsim N^d$. Therefore 
\[
    \|Q_N\|_{\mathfrak{S}_2(L^2(V_2))}^2 \gtrsim \int_{B(y_0,\delta)\times C_N} |q_0(y,\eta)|^2\mathrm{d}y\mathrm{d}\eta \gtrsim N^d.
\]

It remains to observe that by quasi-orthogonality of the family $\{Q_N+R_N\}$ we have 
\[
    \|T\|_{\mathfrak{S}_2(L^2(V_2))}^2 \gtrsim \sum_{N \in 2^{\mathbb{N}}} \|Q_N + R_N\|_{\mathfrak{S}_2(L^2(V_2))}^2 \gtrsim \sum_{N \in 2^{\mathbb{N}}}N^d-N^{d-2\delta} = + \infty,
\]
which is a contradiction. Therefore $q_0=0$.
\end{proof}

It follows from \cref{lemma:contradiction} and \cref{lemma:main-symbol} that for $y$ such that \(\theta_0\circ\kappa^{-1}(y)=1\), and all $\eta \in \Rr^d$ there holds
\[
    |\eta|_{y}^2 = |\Jac \chi(\chi^{-1}(y))^\top\eta|_{\chi^{-1}(y)}^2. 
\]
After evaluation at $y=\chi(0)$ this is nothing but the dual statement of the fact that 
\[
    \Jac X_{t_0 + s,t_0}(x_0) : \left(T_{x_0}M, g(x_0)\right) \longrightarrow \left(T_{X(x_0)}M,g(X(x_0))\right)
\]
is an isometry. Combined with~\cref{app:Killing}, this ends the proof.

\subsection{Proof of Corollaries~\ref{thm.euclidean} and~\ref{thm.main}}

\begin{proof}[Proof of~\cref{thm.euclidean}]
Assume~(iii). We have~\(X_t^{-1}(x) = b(t) + R(t)x\) with~\(R(t) \in O_d(\Rr)\). With the notation of~\cref{subsec:Result}, we get (since the real-valued character is preserved by~\(S(t)\)):
\[
    S(t) \Phi_g^\alpha = \mathrm{Re}\int (1+|\xi|^2)^{-\frac{\alpha}{2}} e^{i (b(t) + R(t)x)\cdot \xi} \mathrm{d}W(\xi).
\]
Process~\(e^{ib(t)\cdot \xi} \mathrm{d}W(\xi)\) has the same law as~\(\mathrm{d}W(\xi)\). Indeed, for all test functions~\(f,g\), we have
\begin{align*}
   \E \left( \overline{\int f(\xi) e^{i (b(t))\cdot \xi} \mathrm{d}W(\xi) } \int g(\eta) e^{i (b(t))\cdot \eta} \mathrm{d}W(\eta)\right) &= \int \overline{f(\xi)} g(\eta) e^{ib(t)\cdot (\eta - \xi)} \delta (\xi - \eta) \mathrm{d}\xi \mathrm{d}\eta \\
   &= \int \overline{f(\xi)} g(\xi) \mathrm{d}\xi. 
\end{align*}
Therefore,
\[
    S(t) \Phi_g \sim_{\mathrm{law}} \mathrm{Re}\int (1+|\xi|^2)^{-\frac{\alpha}{2}} e^{i  (R(t)x)\cdot \xi} \mathrm{d}W(\xi) =  \mathrm{Re}\int (1+|\xi|^2)^{-\frac{\alpha}{2}} e^{i  x\cdot (R(t)^* \xi)} \mathrm{d}W(\xi) .
\]
Wiener processes are invariant under the action of~\(O_d(\Rr)\): indeed, for all test functions~\(f,g\), we have
\[
    \E \left(\overline{\int f\circ R(\xi) \mathrm{d}W(\xi) } \int g\circ R(\eta)  \mathrm{d}W(\eta) \right) = \int \overline{f\circ R(\xi)} g\circ R(\xi)\mathrm{d}\xi = \int \overline{f(\xi)} g(\xi) \mathrm{d}\xi.
\]
We get 
\[
    S(t) \Phi_g \sim_{\mathrm{law}}  \mathrm{Re}\int (1+|R(t)\xi|^2)^{-\frac{\alpha}{2}} e^{i  x\cdot\xi} \mathrm{d}W(\xi) = \mathrm{Re}\int (1+|\xi|^2)^{-\frac{\alpha}{2}} e^{i  x\cdot \xi} \mathrm{d}W(\xi).
\]
Therefore, we have indeed that~(iii) implies~(ii), which in turn implies~(i). 

Assume~(i). We get that~\(u_t\) is a Killing field, which ensures that~\(X_t\) is a global isometry. In other words, for all~\(x,y\in \Rr^d\), we have that
\[
    |X_t(x) - X_t(y)|^2 = |x-y|^2.
\]
We deduce for all~\((x_0,y_0)\in \Rr^{2d}\), all~\((z,v)\in \Rr^2\) and all~\((\xi, \eta) \in \Rr^{2d}\) that
\[
    |X_t(x_0 + z\xi) - X_t(y_0 + v\eta)|^2 = |x_0-y_0+z\xi-v\eta|^2.
\]
We differentiate with respect to~\(z\) and~\(v\) and take the value at~\((z,v) = (0,0)\) to get
\[
    \an{\Jac X_t(x_0) \xi , \Jac X_t(y_0) \eta}=\an{\xi,\eta}. 
\]
Therefore, from this and~\(\Jac X_t(x_0) \in O_d(\Rr)\), we get
\[
    \Jac X_t(x_0) = \Jac X_t(y_0) =: R(t)^*
\]
with~\(R(t) \in O_d(\Rr)\) and thus
\[
X_t^{-1} (x) = R(t) x+ b(t).
\] 
We use~\cref{rema-isom} to conclude.
\end{proof}

\begin{proof}[Proof of~\cref{thm.main}]

The implications \textit{(iii)} $\Longrightarrow$ \textit{(ii)} $\Longrightarrow$ \textit{(i)} are immediate in~\cref{thm.main}. Indeed, if $b(t)=\int_0^tc(\tau)\mathrm{d}\tau$ then the solution to \eqref{eq:transport} for $u(t,x)=c(t)$ with initial data 
\[
    \Phi_g(x)=\sum_{n\in \mathbb{Z}^2_*}a_n\tilde g_n e^{in\cdot x}
\]
is given by 
\[
    S(t)\Phi_g(x)=\sum_{n\in\mathbb{Z}_*^2}a_n\tilde g_n e^{in\cdot (x-b(t))} = \sum_{n\in\mathbb{Z}_*^2}a_n \tilde g_n(t)e^{in\cdot x},
\]
where $\tilde g_n(t)=\tilde g_ne^{-in\cdot b(t)}$. The law of the family~\((\tilde g_n(t))_n\) is the same as the law of the~\((\tilde g_n)_n\): indeed, Gaussian measures are invariant under the action of~\(U(1)\) and
\[
    \E(\tilde g_{-m}(t) \tilde g_n(t)) = \E(\overline{\tilde g_m(t)} \tilde g_n(t)) = e^{ib(t)\cdot(m-n)}\E(\overline{\tilde g_m} \tilde g_n) = \delta_n^m.
\]
It follows that $S(t)_{*}\mu_\alpha=\mu_\alpha$, i.e. we have the invariance of the measure under the flows of \eqref{eq:transport}.

Assume~(i). As for the Euclidean case, we get
\[
    X_t(x) = R(t) x + b(t).
\]
with~\(R(t) \in O_d(\Rr^d)\). But~\(X_t\) must preserve the torus, meaning that~\(R(t) = P_\sigma D\) where~\(P_\sigma\) is the matrix of a permutation and~\(D\) is a diagonal matrix with~\(\pm 1\) on the diagonal. Since~\(X_t\) is continuous, we get that~\(R(t) = R(0) = I_d\) and thus that~\(u(t,\cdot) = c(t)\) with~\(c(t) = \dot b(t)\).
\end{proof}

\appendix

\section{Killing vector fields}\label{app:Killing}

\begin{lemm}\label{lemma:killing}
Let $(M,g)$ be a smooth boundaryless complete Riemannian manifold. Assume that the flow $X_{t,t_0}$ of a vector field $u \in L^1_{\rm loc}(\mathbb{R},\mathcal{C}^{1,\varepsilon}(M))$, as defined in \eqref{eq:defX}, satisfies the following property: for all $t_0 \in \Rr$ and all $x \in M$ there exists $\delta=\delta(t_0,x)>0$ such that for all $t\in \Rr $ satisfying $|t-t_0|\leq \delta$, 
\[
    \Jac X_{t,t_0}(x) : (T_xM,g(x)) \longrightarrow (T_{X_{t,t_0}(x)}M,g(X_{t,t_0}(x)))
\]
is an isometry. 
Then, for almost every $t$, the vector field $u(t)$ is a Killing field, namely: 
\[
    	\nabla_i u_j(t)+\nabla_j u_i(t) = 0.
\]
Conversely, if for almost all~\(t\), we have 
\[
    	\nabla_i u_j(t)+\nabla_j u_i(t) = 0,
\]
then for all~\(t\in \Rr\) and all~\(x\in M\), we have that~\(\Jac X_t (x)\) is an isometry.
\end{lemm}

\begin{proof} Fix $t_0\in \Rr$ and $x \in M$. Define $Y_t(x)=X_{t,t_0}(x)$. In coordinates, we compute the derivative in time of
\[
I_{\mu \nu}(x,t_0,t) = g_{ij}(Y_{t}) \partial_\mu Y_t^i \partial_\nu Y_t^j.
\]
We get for almost every $t$
\begin{equation*}
    \partial_t I_{\mu \nu}(x,t_0,t) = \partial_k g_{ij}(Y_{t}) \partial_t Y_t^k \partial_\mu Y_t^i \partial_\nu Y_t^j + g_{ij}(Y_t) \partial_t\partial_\mu Y_t^i \partial_\nu Y_t^j + g_{ij}(Y_t) \partial_\mu Y_t^i \partial_t\partial_\nu Y_t^j.
\end{equation*}
Using $\partial_t Y_t = u(t,Y_t(x))$ and 
\[
    \partial_t\partial_{\mu} Y_t^i = \partial_{\ell} u^i(t,Y_t)\partial_\mu Y_t^{\ell}, 
\]
we obtain that for almost every $t$ there holds 
\begin{multline*}
    \partial_t I_{\mu \nu}(x,t_0,t) = \\
    \partial_k g_{ij}(Y_t) u^k(t,Y_t) \partial_\mu Y_t^i \partial_\nu Y_t^j + g_{ij}(Y_t) \partial_{\ell} u^i(t,Y_t)\partial_\mu Y_t^{\ell}\partial_\nu Y_t^j + g_{ij}(Y_t) \partial_\mu Y_t^i \partial_{\ell} u^j(t,Y_t)\partial_\nu Y_t^{\ell}.
\end{multline*}
Relabelling indices allows to rewrite the above as 
\[
    \partial_t I_{\mu \nu}(x,t_0,t) = \left(u^k\partial_kg_{ij} + g_{kj}\partial_iu^k + g_{ik}\partial_ju^k\right)\partial_{\mu}Y_t^i \partial_{\nu}Y^j_t,
\]
where $u^k=u^k(t,Y_t)$ and $g_{ij}=g_{ij}(Y_t)$.

The isometry condition writes for~\(t\in (t_0-\delta(t_0,x), t_0+\delta(t_0,x))\)
\[
    g_{\mu \nu}(x) = I_{\mu\nu} (x,t,t_0)
\]
and thus for almost all~\(t\in (-\delta(t_0,x), \delta(t_0,x))\)
\[
\partial_t I_{\mu\nu} (x,t,t_0) = 0.
\]
The matrix $(\partial_\mu Y^i_{t,t_0})_{i,\mu}$ is invertible (it is the matrix of~\(\Jac X_{t,t_0}\)) and therefore we infer 
\[
    u^k\partial_kg_{ij} + g_{kj}\partial_iu^k + g_{ik}\partial_ju^k=0, \qquad i,j \in \{1, \dots, d\}.
\]
Evaluating at $t=t_0$ and renaming $t_0$ as $t$ yields that for almost every $t$ and all $x\in M$ there holds 
\[
    u^k(t,x)\partial_kg_{ij}(x) + g_{kj}(x)\partial_iu^k(t,x) + g_{ik}(x)\partial_ju^k(t,x)=0, \qquad i,j \in \{1, \dots, d\}.
\]
Next, we show that 
\[
    \nabla_{i} u_j(t) + \nabla_j u_i(t) = 0
\]
where 
\[
    \nabla_i u_j(t) = \partial_i u_j(t)  - \Gamma^k_{\;\; i j } u_k(t). 
\]
Lowering indices $u_j(t)=g_{jk}u^k(t)$ gives: 
\[
    \partial_i u_j(t) = \partial_i(g_{j\ell}u^\ell(t)) = (\partial_i g_{j\ell})u^\ell(t)+g_{j\ell}\partial_i u^\ell(t).
\]
It follows that 
\[
    \nabla_i u_j(t) = (\partial_i g_{j\ell})u^\ell(t) + g_{j\ell}\partial_i u^\ell(t) - \Gamma^k_{\;\:ij}g_{k\ell}u^\ell(t).
\]
\[
    \nabla_j u_i(t) = (\partial_j g_{i\ell})u^\ell(t) + g_{i\ell}\partial_j u^\ell(t)- \Gamma^k_{\;\:ji}g_{k\ell}u^\ell(t).
\]
With the identity $\Gamma^k_{\;\:ij}=\Gamma^k_{\;\:ji}$, we now have 
\[
    \nabla_i u_j(t)+\nabla_j u_i(t) =  g_{j\ell}\partial_i u^\ell(t) + g_{i\ell}\partial_j u^\ell(t) + u^\ell(t) (\partial_i g_{j\ell} + \partial_j g_{i\ell} - 2g_{k\ell}\Gamma^k_{\;\;ij}).
\]
It remains to use 
\[
    g_{k\ell}\Gamma^k_{\;\;ij} = \frac{1}{2}(\partial_i g_{j\ell} + \partial_j g_{i\ell} - \partial_\ell g_{ij}),
\]
which implies
\[
	\partial_i g_{j\ell} + \partial_j g_{i\ell} - 2g_{k\ell}\Gamma^k_{ij} = \partial_\ell g_{ij}.
\]
Therefore, 
\[
	\nabla_i u_j(t)+\nabla_j u_i(t) = g_{j\ell}\partial_i u^\ell(t) + g_{i\ell}\partial_j u^\ell(t) + u^\ell(t)\partial_\ell g_{ij}=0,
\]
as claimed. 

Conversely, if~\(\nabla_i u_j(t)+\nabla_j u_i(t) = 0\) for allmost all~\(t\) and all~\(x\in M\), we deduce that
\[
 u^k(t,x)\partial_kg_{ij}(x) + g_{kj}(x)\partial_iu^k(t,x) + g_{ik}(x)\partial_ju^k(t,x)=0, \qquad i,j \in \{1, \dots, d\}.
\]
Recall that
\[
    \partial_t I_{\mu\nu}(x,0,t) = \left(u^k\partial_kg_{ij} + g_{kj}\partial_iu^k + g_{ik}\partial_ju^k\right)\partial_{\mu}X_t^i(x) \partial_{\nu}X^j_t(x),
\]
where $u^k=u^k(t,X_t(x))$ and $g_{ij}=g_{ij}(X_t(x))$. We deduce
\[
    \partial_t I_{\mu\nu}(x,0,t) = 0.
\]
Therefore,
\[
 g_{ij}(X_{t}(x)) \partial_\mu X_t^i(x) \partial_\nu X_t^j(x) = I_{\mu \nu} (x,0,t) = I_{\mu \nu} (x,0,0) = g_{\mu \nu}(x)
\]
which translates as the fact that~\(\Jac X_t(x)\) is an isometry.
\end{proof}

\end{document}